\documentclass[a4paper,notitlepage,twoside,reqno,11pt]{amsart}
\usepackage{anysize}
\marginsize{3.4cm}{3.4cm}{3cm}{3cm}

\usepackage{bbm,pifont,latexsym,dcolumn,indentfirst}
\usepackage[hypertexnames=false,bookmarksopen=true,linktocpage=true,pdfstartview={XYZ null null 1.25}]{hyperref}
\usepackage{amsthm,amsmath,amssymb,amscd,amsfonts,mathrsfs}
\usepackage{color,graphicx,xcolor,graphics,subfigure,extarrows,caption2}
\usepackage{pdfpages,titletoc}

\usepackage{autonum}

\newtheorem{thm}{Theorem}[section]
\newtheorem{lem}[thm]{Lemma}
\newtheorem{cor}[thm]{Corollary}

\newtheorem{thmx}{Theorem}

\theoremstyle{definition}
\newtheorem*{defi}{Definition}
\newtheorem*{ques}{Question}
\newtheorem*{rmk}{Remark}
\newtheorem{prob}{Problem}

\newcommand{\EC}{\widehat{\mathbb{C}}}
\newcommand{\C}{\mathbb{C}}
\newcommand{\D}{\mathbb{D}}
\newcommand{\T}{\mathbb{T}}
\newcommand{\R}{\mathbb{R}}
\newcommand{\Z}{\mathbb{Z}}

\newcommand{\A}{\mathbb{A}}
\newcommand{\Q}{\mathbb{Q}}

\newcommand{\ii}{\textup{i}}

\newcommand{\Int}{\textup{int}}
\newcommand{\Ext}{\textup{ext}}

\newcommand{\im}{\textup{Im\,}}
\newcommand{\Mod}{\textup{mod}}
\newcommand{\area}{\textup{area}}
\newcommand{\HT}{\textup{HT}}

\newcommand{\MA}{\mathcal{A}}
\newcommand{\MB}{\mathcal{B}}

\newcommand{\ME}{\mathcal{E}}
\newcommand{\MF}{\mathcal{F}}

\newcommand{\MJ}{\mathcal{J}}
\newcommand{\MK}{\mathcal{K}}
\newcommand{\MO}{\mathcal{O}}
\newcommand{\MP}{\mathcal{P}}

\newcommand{\MU}{\mathcal{U}}
\newcommand{\MV}{\mathcal{V}}
\newcommand{\MW}{\mathcal{W}}
\newcommand{\PZ}{\mathcal{PZ}}
\newcommand{\PC}{\mathcal{PC}}

\makeatletter\@addtoreset{equation}{section}\makeatother

\begin{document}

\author{Fei Yang}
\address{Department of Mathematics, Nanjing University, Nanjing 210093, P. R. China}
\email{yangfei@nju.edu.cn}

\title[Smooth degenerate Herman rings]{Rational maps with smooth degenerate Herman rings}

\begin{abstract}
We prove the existence of rational maps having smooth degenerate Herman rings. This answers a question of Eremenko affirmatively.
The proof is based on the construction of smooth Siegel disks by Avila, Buff and Ch\'{e}ritat as well as the classical Siegel-to-Herman quasiconformal surgery.
A crucial ingredient in the proof is the surgery's continuity, which relies on the control of the loss of the area of quadratic filled-in Julia sets by Buff and Ch\'{e}ritat.

As a by-product, we prove the existence of rational maps having a nowhere dense Julia set of positive area for which these maps have no irrationally indifferent periodic points, no Herman rings, and are not renormalizable.
\end{abstract}



\date{\today}

\maketitle



\section{Introduction}

\subsection{Backgrounds}

The idea of studying invariant curves of rational maps dates back to 1920s. In the last paragraph of \cite{Fat20b}, Fatou wrote\footnote{Fatou's original text was written in French.} ``It would remain for us to study the \textit{invariant analytic curves} of a rational transformation and whose study is intimately linked to that of the functions studied in this chapter. We hope to return there soon." However, in his published work, Fatou never returned to this topic as far as we know.

After Siegel and Arnold-Herman's remarkable work on the linearization of holomorphic functions \cite{Sie42}, \cite{Her79}, the invariant analytic Jordan curves were known to be exist in Siegel disks and Herman rings.
In 1989, Azarina showed that if a non-linear entire function has an invariant analytic Jordan curve, then either this curve is a circle and the function is conjugate to $z\mapsto z^n$ for some integer $n\geqslant 2$, or it is an invariant curve in a Siegel disk of the function \cite{Aza89}.

A natural question is: besides circles, invariant curves in Siegel disks and Herman rings, are there any other forms of invariant analytic Jordan curves?
Eremenko found two types of analytic Jordan curves which are not circles and invariant under the Latt\'{e}s maps, and moreover, the restriction of the map on each of these curves is not a homeomorphism \cite{Ere12}.

If the invariant Jordan curves are allowed to be \textit{smooth} of some degree but not necessarily analytic, then the related results become richer.
For example, let $f$ be a rational map having the expansion $f(z)=z+a z^2+\mathcal{O}(z^3)$ near the parabolic fixed point $0$ with $a\neq 0$. Then the union of $0$ and the immediate parabolic basin of $0$ contains uncountably many invariant $C^1$-smooth Jordan curves, which are the level curves $\{z\in\EC\,:\,\im \varphi(z)=\text{constants}\}$ of the attracting Fatou coordinate $\varphi$ satisfying $\varphi(f(z))=\varphi(z)+1$. In fact, these curves are analytic everywhere except possibly at $0$ (see \cite[\S 4]{Aza89}).

For rational maps or transcendental meromorphic functions defined on $\C$, the boundaries of rotation domains cannot be analytic curves by Schwarz reflection. However, they can be smooth curves. Indeed, there exist Siegel disks in many families of analytic maps whose boundaries are $C^\infty$-smooth (see \cite{Per97a}, \cite{ABC04}, \cite{Gey08} and \cite{ABC20}), and there also exist Herman rings whose boundaries are $C^\infty$-smooth \cite{Avi03}.
Moreover, Buff and Ch\'{e}ritat proved the existence of quadratic Siegel disks whose boundaries have various degrees of smoothness, e.g., $C^n$ but not $C^{n+1}$ for any positive integer $n$, and $C^0$ but not H\"{o}lder etc \cite{BC07}.
In the following, when we mention ``smooth", it means $C^\infty$-smooth if not otherwise specified.

\begin{defi}
A Jordan curve $\gamma\subset\EC$ is called a \textit{degenerate Herman ring} of a meromorphic function $f$ (rational or transcendental) if $\gamma$ is \textit{not a spherical circle} and satisfies the following properties:
\begin{itemize}
\item $\gamma$ is contained in the Julia set $J(f)$ of $f$;
\item $\gamma$ is not a boundary component of any Siegel disk or Herman ring of $f$;
\item $f(\gamma)=\gamma$ and $f:\gamma\to\gamma$ is conjugate to an irrational rotation.
\end{itemize}
\end{defi}

It seems that the term ``degenerate Herman ring" appeared first in \cite{FH06}, where $\gamma$ is allowed to be a circle.
In fact, there are a lot of rational maps having a circle satisfying the above three properties. As an example, for any irrational number $\alpha\in\R/\Z$, there exists a unique irrational $t=t(\alpha)\in\R/\Z$ such that the unit circle $\T$ is invariant under
\begin{equation}
B_{\alpha}(z):=e^{2\pi\ii t}z^2\frac{z-3}{1-3z},
\end{equation}
and that $B_\alpha:\T\to\T$ is conjugate to the irrational rotation $R_\alpha(\zeta):=e^{2\pi\ii\alpha}\zeta$ (see \cite[p.\,68]{MS93} and \cite{Yoc84b}). Note that $B_\alpha$ has a critical point on $\T$ and it is easy to check that $B_\alpha:\T\to\T$ satisfies the above three properties.

By performing quasiconformal surgery in the Fatou set of $B_\alpha$, e.g., changing the multiplier of the super-attracting fixed point $0$ to $\lambda$ with $0<|\lambda|<1$ (see \cite[\S 4.2]{BF14a}), it is easy to see that the new rational map has a degenerate Herman ring which is a quasi-circle containing a double critical point (It cannot be a circle since the map has lost the symmetry).
This gives another type of invariant Jordan curve which is different from the known examples but the method is somewhat trivial, and moreover, the invariant curve cannot be smooth due to the presence of critical point on the curve.
Hence a natural problem is to find degenerate Herman rings which cannot be obtained from such a quasiconformal deformation.

\subsection{Main results}

In the open problems session of the conference ``On Geometric Complexity of Julia Sets II" held in B\c{e}dlewo, August 2020, Eremenko raised the following question\footnote{The original statement in \cite{Ere20} misses the adjective ``smooth". After communicating with Alexandre Eremenko we realized that a meaningful and difficult question is about the existence of smooth degenerate Herman rings.} (see \cite[Question 1.7, p.\,2]{Ere20}):

\begin{ques}[Eremenko]
Are there smooth degenerate Herman rings?
\end{ques}

Note that if one applies the previous quasiconformal surgery to $B_\alpha$ directly (i.e., changing the multiplier of the origin such that the resulting map loses the symmetry), the smoothness of the circle under a quasiconformal map is not guaranteed.

Combining the idea of constructing smooth Siegel disks by Avila, Buff and Ch\'{e}ritat \cite{ABC04} and the classical Siegel-to-Herman quasiconformal surgery \cite{Shi87}, we give an affirmative answer to Eremenko's question.

\begin{thmx}\label{thm:main-1}
There exist cubic rational maps having a smooth degenerate Herman ring.
\end{thmx}

In fact, Eremenko's question was asked earlier in \cite[p.\,263]{Ere12} in another stronger form: ``Does there exist an invariant \textit{analytic} Jordan curve of a rational map, different from a circle, which is mapped onto itself homeomorphically and intersects the Julia set?" This question is more challenging for degenerate Herman rings, i.e., do there exist \textit{analytic} degenerate Herman rings? Unfortunately the method in this paper cannot give a conclusion on it.

Since the degenerate Herman rings constructed in this paper are smooth, they cannot contain any critical points. This implies that they cannot be obtained by performing quasiconformal deformation on the cubic Blaschke product $B_\alpha$.

\medskip
After Theorem \ref{thm:main-1}, a question is about the area (i.e., $2$-dimensional Lebesgue measure) of the Julia sets of the maps therein.
We will see that these maps produce Julia sets with positive area that have never appeared before.

Fatou began to study the area of Julia sets of rational maps as early as the beginning of the last century \cite{Fat19}. It was not until 2005 that Buff and Ch\'{e}ritat constructed the first nowhere dense Julia sets of rational maps with positive area\footnote{See the references in \cite{BC12} and \cite{AL22} for the progress of Julia sets with area zero, and \cite{Buf97} for an attempt to study the positive area problem of Julia sets by Nowicki and van Strien for Fibonacci maps of high degree.} \cite{BC12}.
As far as we know, all the known rational maps having a nowhere dense Julia set with positive area belong to the following cases: either they have an irrationally indifferent periodic point, or they are infinitely renormalizable (see \cite{BC12}, \cite{AL22}, \cite{DL23a} and also \cite{FY20}, \cite{QQ20}). Additionally, by a standard quasiconformal surgery, one can construct rational maps having a Herman ring and a Julia set with positive area.

In this paper, we prove that there exist rational maps having a Julia set of positive area which do not belong to any one of the above cases.

\begin{thmx}\label{thm:main-2}
There exist cubic rational maps having a nowhere dense Julia set of positive area. These maps have no irrationally indifferent periodic points, no Herman rings, and are not renormalizable.
\end{thmx}

In fact, we will prove the existence of a rational map such that Theorems \ref{thm:main-1} and \ref{thm:main-2} hold at the same time (see \S\ref{sec:proof-ThmA}).
Intuitively, the rational maps therein can be seen to be obtained by ``pasting" two quadratic Siegel polynomials along their Siegel disk boundaries, which can be also seen as some form of ``tuning" \cite[Chap.\,V]{DH85b} and ``simultaneous  uniformization" \cite{McM85}.

\medskip
Recently, another method of constructing non-trivial degenerate Herman rings (i.e., not by a quasiconformal deformation of Blaschke products) was given by Lim, based on the study of \textit{a priori} bounds of bounded type Herman rings \cite{Lim23a}. The degenerate Herman rings constructed there are quasi-circles passing through critical points (hence are not smooth) and the corresponding Julia sets likely to have area zero.

\subsection{Idea of the proofs}

The proof of Theorem \ref{thm:main-1} is based on two main ingredients: a perturbation lemma in Avila-Buff-Ch\'{e}ritat's construction of smooth Siegel disks
(see Lemma \ref{lem:ABC04}) and a standard quasiconformal surgery which turns Siegel disks to Herman rings. We ``paste" two quadratic Siegel polynomials $P_\alpha$ and $P_{-\alpha}$ together such that the resulting cubic rational map $Q_\alpha$ has a fixed Herman ring and is not a Blaschke product, by destroying the symmetry, where $P_{\alpha}(z):=e^{2\pi\ii\alpha}z+z^2$ with $\alpha\in\R\setminus\Q$. By choosing a suitable sequence $(\alpha_n)_{n\geqslant 0}$ of irrational numbers such that the conformal radius of the Siegel disk of $P_{\alpha_n}$ decreases strictly to a positive number and that $\alpha_n\to\alpha'$ as $n\to\infty$, where $\alpha'\in\R\setminus\Q$, we prove that the limit cubic rational map $Q_{\alpha'}$ cannot keep any spherical circle invariant and has a smooth degenerate Herman ring. Moreover, the construction leads to $\alpha'$ being a Brjuno number.

There are two crucial points in the proof: the first is the continuity of the surgery and the second is the convergence of smooth Jordan curves in the Herman rings as we take limit (the conformal modulus of the Herman ring tends to zero).
To prove the continuity of the surgery, we need to prove the convergence of a sequence of Beltrami coefficients. To this end, we use a result of Buff and Ch\'{e}ritat \cite{BC12} to control the loss of area of quadratic filled-in Julia sets (see Lemma \ref{lem:control-loss}). In particular, all the irrational numbers need to be chosen as sufficiently high type.
For the second crucial point, we consider the convergence of smooth curves in the Fr\'{e}chet space $C^\infty(\R/\Z,\C)$, which consists of all $C^\infty$-function from $\R/\Z$ to $\C$. This is inspired by \cite{ABC04}.

\medskip
The smooth degenerate Herman ring in Theorem \ref{thm:main-1} divides the Riemann sphere into two Jordan disks in which the limit cubic rational map $Q_{\alpha'}$ is proved to be quasiconformally conjugate to the quadratic Siegel polynomial $P_{\alpha'}$ and $P_{-\alpha'}$ respectively.
By choosing the sequence $(\alpha_n)_{n\geqslant 0}$ of irrational numbers suitably, we can control the area of the Julia set of $P_{\alpha'}$ without affecting Theorem \ref{thm:main-1}. In particular, there exists an irrational number $\alpha'$ and a cubic rational map $Q_{\alpha'}$ such that Theorems \ref{thm:main-1} and \ref{thm:main-2} hold at the same time.

\medskip
In the last section we supplement some results and list some problems about (degenerate) Herman rings.

\medskip
\textit{Remark on the proofs}. One may wonder whether it is possible to prove \ref{thm:main-1} directly by quasiconformal deformation, but not using Siegel-to-Herman surgery, just like the argument in \cite{Avi03}. However, to apply a similar argument, one needs to find a slice of the space of rational maps such that the conformal moduli of the non-symmetric Herman rings can be controlled well under the perturbation of rotation number, which turns out to be difficult. In fact, this is a major difference from \cite{Avi03}, where the space of cubic Blaschke products is considered and one has such a slice automatically.
But we would like to mention that for Theorem \ref{thm:main-2}, the rational maps can be chosen as cubic Blaschke products (see \S\ref{subsec:Blaschke}).

\medskip
\noindent\textbf{Acknowledgements.}
The author is very grateful to two anonymous reviewers for their very insightful and detailed comments, suggestions and corrections.
He would also like to thank Alexandre Eremenko and Lasse Rempe for valuable comments, and Hongyu Qu for helpful discussions.
This work was supported by NSFC (Grant Nos.\ 12222107 and 12071210).

\section{Siegel-to-Herman surgery and rigidity}\label{sec:Siegel-to-Herman}

In this section, we perform the quasiconformal surgery which pastes two Siegel disks of quadratic polynomials to obtain a Herman ring of a cubic rational map. Under Petersen-Zakeri's condition we prove that the surgery has some rigidity.

\subsection{Siegel-to-Herman surgery}\label{subsec:S-to-H}

In this paper, for $r>0$ we denote $\D_r:=\{z\in\C: |z|<r\}$ and $\T_r:=\partial\D_r$. In particular, $\D:=\D_1$ and $\T:=\T_1$. For any $0<r_1<r_2<+\infty$, we denote $\A_{r_1,r_2}:=\{z\in\C: r_1<|z|<r_2\}$.

Let $[a_0;a_1,a_2,\cdots,a_n,\cdots]$ be the continued fraction expansion of an irrational number $\alpha\in \R\setminus\Q$, where $a_n\geqslant 1$ for all $n\geqslant 1$. The rational numbers \begin{equation}\label{equ:q-n}
a_0+p_n/q_n := a_0+[0; a_1,\cdots , a_n], \quad n\geqslant 1,
\end{equation}
are the convergents of $\alpha$, where $p_n$ and $q_n$ are coprime positive integers.
Suppose $\alpha\in \R\setminus\Q$ is a \textit{Brjuno number}, i.e.,
\begin{equation}
\sum_{n\geqslant 1}\frac{\log q_{n+1}}{q_n}<+\infty.
\end{equation}
According to \cite{Sie42} and \cite{Brj71}, every holomorphic germ $f(z)=e^{2\pi\ii\alpha}z+\MO(z^2)$ is locally linearizable at $0$, i.e., $f$ is holomorphically conjugate to the rigid rotation $R_\alpha(\zeta):=e^{2\pi\ii\alpha}\zeta$ in a neighborhood of $0$. The largest $f$-invariant domain containing $0$ in which $f$ is conjugate to $R_\alpha$ is called the \textit{Siegel disk} of $f$ centered at $0$.
In particular, if $\alpha\in\R\setminus\Q$ is Brjuno, then the quadratic polynomial
\begin{equation}
P_{\alpha}(z):=e^{2\pi \ii\alpha}z+z^2
\end{equation}
has a \textit{Siegel disk} $\Delta_{\alpha}$ centered at the origin.
There exist a unique $r_{\alpha}>0$ and a unique conformal map $\phi_{\alpha}:\mathbb{D}_{r_{\alpha}}\to\Delta_{\alpha}$ with $\phi_\alpha(0)=0$ and $\phi_\alpha'(0)=1$, such that
\begin{equation}\label{equ:linearization}
\phi_{\alpha}\circ R_\alpha(\zeta)=P_{\alpha}\circ\phi_{\alpha}(\zeta) \text{\quad for all } \zeta\in\mathbb{D}_{r_{\alpha}}.
\end{equation}
The number $r_{\alpha}>0$ is called the \textit{conformal radius} of the Siegel disk $\Delta_{\alpha}$.
Yoccoz proved that if $\alpha$ is not a Brjuno number, then $P_\alpha$ cannot be locally linearizable at $0$ \cite{Yoc95}.
If $\alpha\in\R\setminus\Q$ is not a Brjuno number, we denote $\Delta_\alpha=\emptyset$ and $r_\alpha=0$.

\medskip
In this subsection, we perform the quasiconformal surgery which turns Siegel disks to Herman rings. The original idea can be found in \cite{Shi87} (see also \cite[\S 7.3]{BF14a} and \cite{BFGH05}).
Let $\alpha$ be a Brjuno number and $0<r<r_\alpha$. Then
\begin{equation}
\gamma_{\alpha,r}:=\phi_{\alpha}(\T_r)
\end{equation}
is a $P_\alpha$-invariant analytic Jordan curve in $\Delta_\alpha$. In the following,
\begin{itemize}
\item For a Jordan curve $\gamma$ (resp. an annulus $A$) in $\C$, we use $\gamma^{\Ext}$ (resp. $A^{\Ext}$) to denote the connected component of $\EC\setminus\gamma$  (resp. $\widehat{\mathbb{C}} \setminus A$) containing $\infty$ and use $\gamma^{\Int}$  (resp. $A^{\Int}$) to denote the other component of $\widehat{\mathbb{C}} \setminus \gamma$ (resp. $\widehat{\mathbb{C}} \setminus A$);
\item For two Jordan curves $\gamma_1$ and $\gamma_2$ in $\C$ satisfying $\gamma_1\subset\gamma_2^\Int$, we use $A(\gamma_1,\gamma_2)\subset\C$ to denote the open annulus bounded by $\gamma_1$ and $\gamma_2$; and
\item For an annulus $A$ in $\EC$, we use $\Mod(A)$ to denote its \textit{conformal modulus}.
\end{itemize}

\begin{lem}\label{lem:annulus}
For any Brjuno number $\alpha$ and $0<r<r'<r_\alpha$, we have
\begin{equation}
\eta_r(\gamma_{-\alpha,r'}^\Ext)\subset\gamma_{\alpha,r/2}^\Int, \text{\quad where }\eta_r(z):=\tfrac{r^2}{32z}.
\end{equation}
\end{lem}

\begin{proof}
Note that $\overline{P_{-\alpha}(\overline{z})}=P_\alpha(z)$. By \eqref{equ:linearization} we have $\overline{\phi_{\alpha}\circ \overline{R_{-\alpha}(\zeta)}}=P_{-\alpha}\circ\overline{\phi_{\alpha}(\overline{\zeta})}$ for all $\zeta\in\mathbb{D}_{r_{\alpha}}$.
Since $\phi_{\alpha}:\mathbb{D}_{r_{\alpha}}\to\Delta_{\alpha}$ is the conformal map with $\phi_\alpha(0)=0$ and $\phi_\alpha'(0)=1$, we have
\begin{equation}\label{equ:phi-minus-alpha}
\phi_{-\alpha}(\zeta)=\overline{\phi_\alpha(\overline{\zeta})}  \text{\quad and\quad} r_{-\alpha}=r_{\alpha}.
\end{equation}
By the Koebe distortion theorem (see \cite[Theorem 1.6, p.\,21]{Pom75}), we have
\begin{equation}
\gamma_{\alpha,r/2}^\Int=\phi_\alpha(\D_{r/2})\supset \D_{r/8} \text{\quad and\quad}
\gamma_{-\alpha,r}^\Ext=\EC\setminus\overline{\phi_{-\alpha}(\D_r)}\subset \EC\setminus\overline{\D}_{r/4}.
\end{equation}
This implies that $\eta_r(\gamma_{-\alpha,r'}^\Ext)\subset\eta_r(\gamma_{-\alpha,r}^\Ext)\subset\gamma_{\alpha,r/2}^\Int$.
\end{proof}

By Lemma \ref{lem:annulus}, $\gamma_{\alpha,r}^\Int\setminus \overline{\eta_r(\gamma_{-\alpha,r'}^\Ext)}$ is an annulus whose boundary consists of two analytic Jordan curves $\gamma_{\alpha,r}$ and $\eta_r(\gamma_{-\alpha,r'})$.
Note that $\Mod\big(A(\gamma_{\alpha,r},\partial\Delta_\alpha)\big)=\frac{1}{2\pi}\log\frac{r_\alpha}{r}$ and $\Mod\big(A(\gamma_{-\alpha,r'},\partial\Delta_{-\alpha})\big)=\frac{1}{2\pi}\log\frac{r_\alpha}{r'}$.
For any $\theta_1,\theta_2\in\R$, the following two conformal maps
\begin{equation}\label{equ:psi-alpha-pm}
\begin{split}
\psi_{\alpha,+}(\zeta):=\phi_\alpha\big(r e^{\ii\theta_1}\zeta\big): &~ \A_{1,r_\alpha/r} \to A(\gamma_{\alpha,r},\partial\Delta_\alpha) \text{\quad and\quad} \\
\psi_{\alpha,-}(\zeta):= \eta_r\circ\phi_{-\alpha}\Big(\tfrac{ r}{e^{\ii\theta_2}\zeta}\Big): &~ \A_{r/r_\alpha,r/r'} \to A(\eta_r(\partial\Delta_{-\alpha}),\eta_r(\gamma_{-\alpha,r'}))
\end{split}
\end{equation}
satisfy
\begin{equation}\label{equ:psi-pm}
\begin{split}
\psi_{\alpha,+}^{-1}\circ P_{\alpha}\circ\psi_{\alpha,+}(\zeta)=&~R_\alpha(\zeta) \text{\quad for all } \zeta\in\A_{1,r_\alpha/r} \text{\quad and\quad} \\
\psi_{\alpha,-}^{-1}\circ (\eta_r\circ P_{-\alpha}\circ\eta_r^{-1})\circ\psi_{\alpha,-}(\zeta)=&~R_\alpha(\zeta) \text{\quad for all } \zeta\in\A_{r/r_\alpha,r/r'}.
\end{split}
\end{equation}
Note that any conformal isomorphism which maps the annulus $\A_{r_1,r_2}$ to itself can only be a rigid rotation $\zeta\mapsto e^{\ii\theta}\zeta$ or have the form $\zeta\mapsto r_1 r_2 e^{\ii\theta}/\zeta$ by Schwarz reflection, where $\theta\in\R$.
Hence the conformal maps $\psi_{\alpha,+}:\A_{1,r_\alpha/r} \to A(\gamma_{\alpha,r},\partial\Delta_\alpha)$ and
$\psi_{\alpha,-}:\A_{r/r_\alpha,r/r'} \to A(\eta_r(\partial\Delta_{-\alpha}),\eta_r(\gamma_{-\alpha,r'}))$ satisfying \eqref{equ:psi-pm} must have the form \eqref{equ:psi-alpha-pm}.
Note that we shall perform Shishikura's Siegel-to-Herman surgery and the parameters $\theta_1$ and $\theta_2$ will be used to measure the relative position of the two critical points on the boundary of the Herman ring (see \S\ref{subsec:rigidity}).

By Schwarz reflection, $\psi_{\alpha,+}$ and $\psi_{\alpha,-}$ can be analytically extended to a neighborhood of the boundary components $\T$ and $\T_{r/r'}$ respectively. By \cite[Lemma 2.22 and Remark 2.23]{BF14a}, there exists an interpolation
\begin{equation}\label{equ:psi-0}
\psi_{\alpha,0}:\overline{\A}_{r/r',1} \to \overline{A(\eta_r(\gamma_{-\alpha,r'}),\gamma_{\alpha,r})}
\end{equation}
such that $\Psi_\alpha: \A_{r/r_\alpha,r_\alpha/r} \to A(\eta_r(\partial\Delta_{-\alpha}),\partial\Delta_\alpha)$ is a $C^\infty$-diffeomorphism, where
\begin{equation}\label{equ:Psi}
\Psi_\alpha(\zeta):=
\left\{
\begin{array}{ll}
\psi_{\alpha,+}(\zeta)  &~~~~~~~\text{if}~\zeta\in \A_{1,r_\alpha/r}, \\
\psi_{\alpha,-}(\zeta) &~~~~~~\text{if}~\zeta\in  \A_{r/r_\alpha,r/r'}, \\
\psi_{\alpha,0}(\zeta) &~~~~~~\text{if}~\zeta\in \overline{\A}_{r/r',1}.
\end{array}
\right.
\end{equation}
Note that the interpolation $\psi_{\alpha,0}$ is not unique. We fix a choice here.
See Figure \ref{Fig-surgery}.

\begin{figure}[!htpb]
  \setlength{\unitlength}{1mm}
  \centering
  \includegraphics[width=0.9\textwidth]{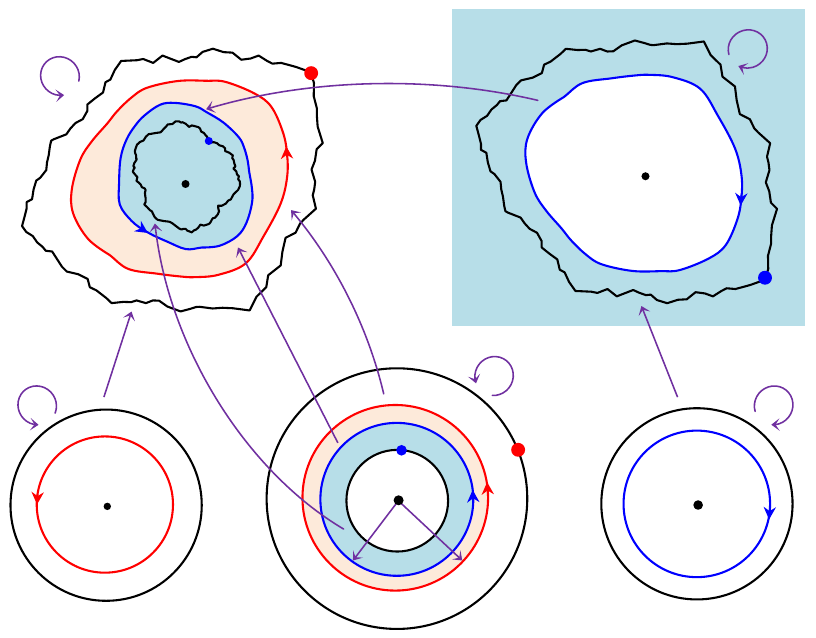}
  \put(-111,92.5){$\Delta_\alpha$}
  \put(-66,89){$\eta_r$}
  \put(-125,92){$P_\alpha$}
  \put(-97,70){$0$}
  \put(-84.5,81){$\gamma_{\alpha,r}$}
  \put(-77.5,90){$\omega_\alpha$}
  \put(-105,75){\small{$\eta_r(\omega_{-\alpha})$}}
  \put(-125,41){$R_{\alpha}$}
  \put(-109.5,19.5){$0$}
  \put(-116,13){$\D_{r}$}
  \put(-101.5,5){$\D_{r_\alpha}$}
  \put(-115,44){$\phi_{\alpha}$}
  \put(-49,5){$\D_{r_\alpha/r}$}
  \put(-69,24.5){$\D_{r/r_\alpha}$}
  \put(-71,19){$\tfrac{r}{r'}$}
  \put(-62,18.8){$1$}
  \put(-66.2,17){$0$}  
  \put(-95,41){$\psi_{\alpha,-}$}
  \put(-82,47){$\psi_{\alpha,0}$}
  \put(-73,55){$\psi_{\alpha,+}$}
  \put(-46.6,41){$R_\alpha$}
  \put(-8.5,95){$P_{-\alpha}$}
  \put(-50,57){$\Delta_{-\alpha}$}
  \put(-9.5,52){$\omega_{-\alpha}$}
  \put(-25,71){$0$}
  \put(-30,85){$\gamma_{-\alpha,r'}$}
  \put(-9.5,41){$R_{-\alpha}$}
  \put(-16.5,19.5){$0$}
  \put(-24,13){$\D_{r'}$}
  \put(-9,5){$\D_{r_{\alpha}}$}
  \put(-22.8,43.5){$\phi_{-\alpha}$}
  \caption{A sketch of the Siegel-to-Herman surgery construction, where $0<r<r'<r_\alpha$. The critical points $\omega_\alpha$ and $\eta_r(\omega_{-\alpha})$ are marked and their relative position will be determined in \S\ref{subsec:rigidity} under the assumption that $\alpha$ is of Petersen-Zakeri type.}
  \label{Fig-surgery}
\end{figure}
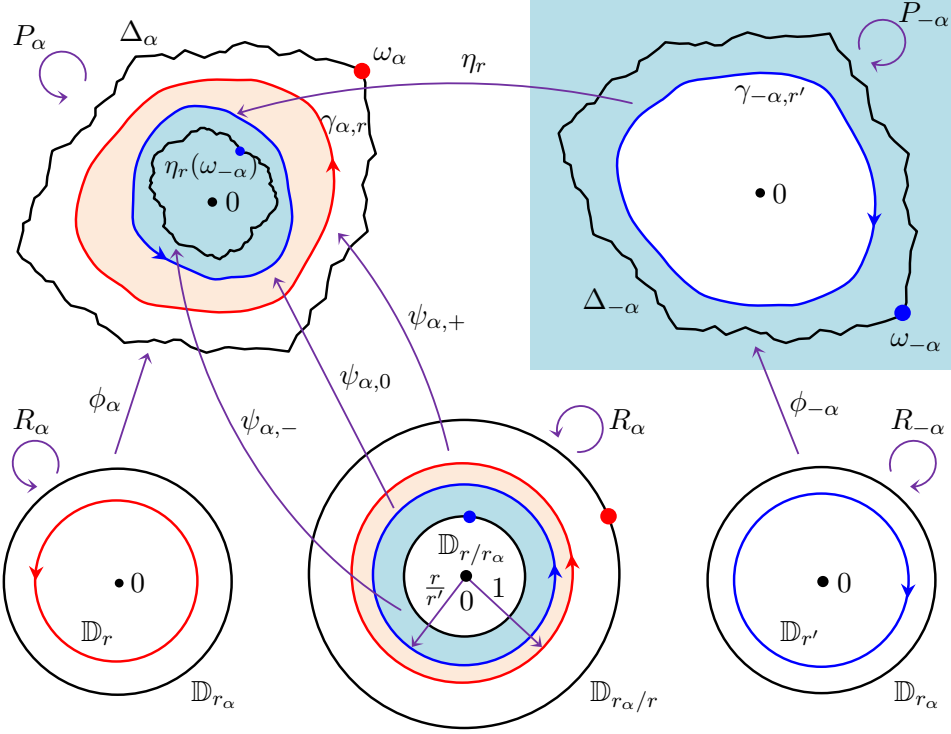

\medskip
Define
\begin{equation}\label{equ:F-quasiregular}
F_\alpha(z):=
\left\{
\begin{array}{ll}
P_\alpha(z)  &~~~~~~~\text{if}~z\in \gamma_{\alpha,r}^\Ext, \\
\eta_r\circ P_{-\alpha}\circ\eta_r^{-1}(z) &~~~~~~\text{if}~z\in \eta_r(\gamma_{-\alpha,r'}^\Ext), \\
\psi_{\alpha,0}\circ R_\alpha\circ\psi_{\alpha,0}^{-1}(z) &~~~~~~\text{if}~z\in \overline{A(\eta_r(\gamma_{-\alpha,r'}),\gamma_{\alpha,r})}.
\end{array}
\right.
\end{equation}
Let $\omega_\alpha:=-e^{2\pi\ii\alpha}/2$ (resp. $\omega_{-\alpha}$) be the critical point of $P_\alpha$ (resp. $P_{-\alpha}$).
We have the following result.

\begin{lem}\label{lem:surgery-S-to-H}
For any Brjuno number $\alpha$, there exists a quasiconformal map $\Phi_\alpha:\EC\to\EC$ satisfying $\Phi_\alpha(0)=0$, $\Phi_\alpha(\infty)=\infty$ and $\Phi_\alpha(\omega_\alpha)=1$ such that
\begin{equation}\label{equ:Q-alpha}
Q_\alpha(z):=\Phi_\alpha\circ F_\alpha\circ\Phi_\alpha^{-1}(z)=u z^2\frac{z-a}{1-\tfrac{2a-3}{a-2} z},
\end{equation}
where $u\in\C\setminus\{0\}$ and $a\in\C\setminus\big\{0,1,\frac{3}{2},2,3\big\}$. The map $Q_\alpha$ has four critical points $\big\{0,1,\infty,c=\frac{a(a-2)}{2a-3}\big\}$ and a fixed Herman ring $A_\alpha$ with rotation number $\alpha$ and modulus $\frac{1}{\pi}\log\frac{r_\alpha}{r}$.
\end{lem}

\begin{proof}
By construction, $F_\alpha:\EC\to\EC$ is a $C^\infty$-branched covering and hence a quasi-regular map which is analytic outside the closure of $A(\eta_r(\gamma_{-\alpha,r'}),\gamma_{\alpha,r})$.
Now we pull back the standard complex structure $\sigma_0$ on $\A_{r/r_\alpha,r_\alpha/r}$ to define an almost complex structure $\sigma_\alpha:=(\Psi_\alpha^{-1})^*(\sigma_0)$ in $A_\alpha':=A(\eta_r(\partial\Delta_{-\alpha}),\partial\Delta_\alpha)$. By \eqref{equ:psi-pm} and \eqref{equ:F-quasiregular}, $\Psi_\alpha^{-1}:A_\alpha'\to\A_{r/r_\alpha,r_\alpha/r}$ is a quasiconformal map which conjugates $F_\alpha: A_\alpha'\to A_\alpha'$ to $R_\alpha:\A_{r/r_\alpha,r_\alpha/r}\to \A_{r/r_\alpha,r_\alpha/r}$. Since $\sigma_0$ is invariant under $R_\alpha$, it follows that $\sigma_\alpha$ is invariant under $F_\alpha: A_\alpha'\to A_\alpha'$.

Next, we extend $\sigma_\alpha$ to $F_\alpha^{-n}(A_\alpha')$ by defining $\sigma_\alpha:=(F_\alpha^{\circ n})^*(\sigma_\alpha)$ in $F_\alpha^{-n}(A_\alpha')\setminus F_\alpha^{-(n-1)}(A_\alpha')$ for all $n\geqslant 1$. In the rest we define $\sigma_\alpha:=\sigma_0$. Note that $F_\alpha:\EC\to\EC$ is analytic outside the closure of $A(\eta_r(\gamma_{-\alpha,r'}),\gamma_{\alpha,r})$. This implies that $\sigma_\alpha$ has uniformly bounded dilatation. By the measurable Riemann mapping theorem \cite{AB60}, there exists a quasiconformal mapping $\Phi_\alpha:\EC\to\EC$ integrating the almost complex structure $\sigma_\alpha$ such that $\Phi_\alpha^*(\sigma_0)=\sigma_\alpha$ and $\Phi_\alpha:(\EC,\sigma_\alpha)\rightarrow (\EC,\sigma_0)$ is an analytic isomorphism. Hence $Q_\alpha:=\Phi_\alpha\circ F_\alpha\circ \Phi_\alpha^{-1}:(\EC,\sigma_0)\to(\EC,\sigma_0)$ is a rational map.
Since $\sigma_\alpha=(\Psi_\alpha^{-1})^*(\sigma_0)=\Phi_\alpha^*(\sigma_0)$ in $A_\alpha'$, it follows that
\begin{equation}\label{equ:Psi-Phi}
\Phi_\alpha\circ\Psi_\alpha:\A_{r/r_\alpha,r_\alpha/r} \to \Phi_\alpha(A_\alpha')
\end{equation}
is a conformal isomorphism which conjugates the rigid rotation $R_\alpha$ in $\A_{r/r_\alpha,r_\alpha/r}$ to $Q_\alpha:\Phi_\alpha(A_\alpha')\to \Phi_\alpha(A_\alpha')$. Hence $Q_\alpha$ has a fixed Herman ring $A_\alpha:=\Phi_\alpha(A_\alpha')$ with rotation number $\alpha$ and modulus $\frac{1}{\pi}\log\frac{r_\alpha}{r}$.

We assume that the quasiconformal mapping $\Phi_\alpha:\EC\to\EC$ is normalized by $\Phi_\alpha(0)=0$, $\Phi_\alpha(\infty)=\infty$ and $\Phi_\alpha(\omega_\alpha)=1$.
Then $Q_\alpha$ is a cubic rational map having two super-attracting fixed points $0$ and $\infty$, and two different critical points $1$ and $c:=\Phi_\alpha(\eta_r(\omega_{-\alpha}))$ in $\C$.
Hence $Q_\alpha$ has the form $Q_\alpha(z)=u z^2 (z-a)/(1-bz)$, where $u,a,b\in\C\setminus\{0\}$. A direct calculation shows that
\begin{equation}
Q_\alpha'(z)=\frac{uz}{(1-bz)^2}\big(-2b z^2+(3+ab)z-2a\big).
\end{equation}
Since $Q_\alpha$ has two critical points $1$ and $c$, we conclude that $b=\frac{2a-3}{a-2}$, $c=\frac{a(a-2)}{2a-3}$ and hence $Q_\alpha$ has the formula \eqref{equ:Q-alpha}. To guarantee that $Q_\alpha$ is a cubic rational map with a finite pole, the parameter $a$ must lie in $\C\setminus\big\{0,1,\frac{3}{2},2,3\big\}$.
\end{proof}

The above surgery construction depends on the parameters $\alpha$, $r$, $r'$ and the quasiconformal interpolation $\psi_{\alpha,0}$ (note that $\psi_{\alpha,0}$ depends on the conformal mappings $\psi_{\alpha,+}$ and $\psi_{\alpha,-}$). Hence a priori, as these parameters and maps vary, one may obtain different $Q_\alpha$'s.
In the rest of this section, we establish a rigidity result about the surgery, i.e., different ways of pastings and interpolations may induce the same $Q_\alpha$.

\subsection{Lebesgue measure of Julia sets}\label{subsec:zero-area}

For any rational map $f:\EC\to\EC$, we use $J(f)$ and $F(f)$, respectively, to denote the \textit{Julia set} and \textit{Fatou set} of $f$.
Let $Q_\alpha$ be the cubic rational map having a fixed Herman ring $A_\alpha$ obtained in Lemma \ref{lem:surgery-S-to-H}.
We use $B_\alpha^\infty$ and $B_\alpha^0$ to denote the \textit{immediate} super-attracting basins of $\infty$ and $0$ respectively (with respect to $Q_\alpha$).
Then
\begin{equation}\label{equ:J-alpha-0-infty}
J_\alpha^\infty:=\partial B_\alpha^\infty \text{\quad and\quad} J_\alpha^0:=\partial B_\alpha^0
\end{equation}
are quasiconformally homeomorphic to $J(P_\alpha)$ and $J(P_{-\alpha})$ respectively and they are both connected.
Here we say that a compact subset $X\subset\EC$ is quasiconformally homeomorphic to another $Y$ if there exists a quasiconformal map $\varphi$ defined in a neighborhood of $X$ such that $\varphi(X)=Y$.
We denote by $\MB_\alpha(\infty)$ and $\MB_\alpha(0)$ the super-attracting basins of $\infty$ and $0$ respectively.
Define
\begin{equation}\label{equ:J-Q-alplha-decom}
\MJ_\alpha^1:=\bigcup_{m\geqslant 0}Q_\alpha^{-m}(J_\alpha^\infty\cup J_\alpha^0) \text{\quad and\quad} \MJ_\alpha^2:=J(Q_\alpha)\setminus \MJ_\alpha^1.
\end{equation}

 \begin{lem}\label{lem:zero-area}
 For any Brjuno number $\alpha$, we have
 \begin{enumerate}
\item $F(Q_\alpha)= \MB_\alpha(\infty)\cup\MB_\alpha(0)\cup\MA_\alpha$, where $\MA_\alpha:=\bigcup_{m\geqslant 0}Q_\alpha^{-m}(A_\alpha)$. Moreover, every component of $\MB_\alpha(\infty)\cup\MB_\alpha(0)$ is simply connected, and every component $A$ of $\MA_\alpha$ is an annulus satisfying $\Mod(A)=\Mod(A_\alpha)$;
\item Each component of $\MJ_\alpha^1$ is quasiconformally homeomorphic to $J(P_\alpha)$ or $J(P_{-\alpha})$;
\item $\MJ_\alpha^2$ is a totally disconnected set having area zero.
\end{enumerate}
 \end{lem}

\begin{proof}
(a) By the surgery construction, $J_\alpha^\infty$ and $J_\alpha^0$ are forward invariant Julia components of $Q_\alpha$, and
\begin{equation}
Q_\alpha: J_\alpha^\infty\to J_\alpha^\infty \text{\quad and\quad} Q_\alpha: J_\alpha^0\to J_\alpha^0
\end{equation}
are quasiconformally conjugate to $P_\alpha|_{J(P_\alpha)}$ and $P_{-\alpha}|_{J(P_{-\alpha})}$ respectively.
In particular, the forward orbit $\MO^+(1)$ of the critical point $1$ is contained in $J_\alpha^\infty$ and the forward orbit $\MO^+(c)$ of another critical point $c=\frac{a(a-2)}{2a-3}$ is contained in $J_\alpha^0$.
Moreover, if $U\not\in \{B_\alpha^\infty, B_\alpha^0\}$ is a Fatou component of $Q_\alpha$ satisfying $\partial U \cap (J_\alpha^\infty\cup J_\alpha^0)\neq\emptyset$, then there exists $m\geqslant 0$ such that $Q_\alpha^{\circ m}(U)=A_\alpha$.

Note that the remaining critical points of $Q_\alpha$ are $0$ and $\infty$, which are super-attracting fixed points. This implies that $Q_\alpha$ has exactly $3$ periodic Fatou components $B_\alpha^\infty$, $B_\alpha^0$ and $A_\alpha$ by the Fatou-Shishikura inequality \cite[Corollary 2]{Shi87}. Therefore, for any Fatou component $U$ of $Q_\alpha$, there exists $m\geqslant 0$ such that $Q_\alpha^{\circ m}(U)=B_\alpha^\infty$, $B_\alpha^0$ or $A_\alpha$.

For an annulus $A$ in $\C$, recall that $A^{\Ext}$ is the connected component of $\widehat{\mathbb{C}} \setminus A$ containing $\infty$ and $A^{\Int}$ is the other component of $\EC\setminus A$.
Then $A\cup A^{\Int}$ and $A\cup A^{\Ext}$ are simply connected domains. Note that
\begin{equation}\label{equ:A-alpha-pm}
Q_\alpha^{-1}(A_\alpha)=A_\alpha\cup A_\alpha^+\cup A_\alpha^-,
\end{equation}
where $A_\alpha^+$ and $A_\alpha^-$ are annuli contained in $A_\alpha^\Ext$ and $A_\alpha^\Int$ respectively. Then
\begin{equation}\label{equ:Q-alpha-pm}
Q_\alpha: A_\alpha^+\cup (A_\alpha^+)^\Int\to A_\alpha\cup A_\alpha^{\Int} \text{\quad and\quad}
Q_\alpha: A_\alpha^-\cup (A_\alpha^-)^\Int\to A_\alpha\cup A_\alpha^{\Ext}
\end{equation}
are conformal maps.
Let $B_\alpha^+$ and $B_\alpha^-$ be the unique Fatou components satisfying
\begin{equation}\label{equ:B-alpha-pm}
\begin{split}
B_\alpha^+\subset (A_\alpha^+)^\Int, & \quad Q_\alpha^{-1}(B_\alpha^0)=B_\alpha^0\cup B_\alpha^+ \text{\quad and} \\
B_\alpha^-\subset (A_\alpha^-)^\Int,  & \quad Q_\alpha^{-1}(B_\alpha^\infty)=B_\alpha^\infty\cup B_\alpha^-.
\end{split}
\end{equation}
Then $B_\alpha^+$ contains the preimage of $0$, at $a$, and $B_\alpha^-$ contains the pole $\frac{1}{b}=\frac{a-2}{2a-3}$.
Note that $A_\alpha^+\cup (A_\alpha^+)^\Int$ and $A_\alpha^-\cup (A_\alpha^-)^\Int$ are disjoint from the forward orbits of the critical points of $Q_\alpha$.
Part (a) follows immediately.

\medskip
(b) Let $J_\alpha^+$ and $J_\alpha^-$ be the unique Julia components satisfying
\begin{equation}\label{equ:J-alpha-pm}
\begin{split}
J_\alpha^+\subset (A_\alpha^+)^\Int, & \quad Q_\alpha^{-1}(J_\alpha^0)=J_\alpha^0\cup J_\alpha^+ \text{\quad and} \\
J_\alpha^-\subset (A_\alpha^-)^\Int,  & \quad Q_\alpha^{-1}(J_\alpha^\infty)=J_\alpha^\infty\cup J_\alpha^-.
\end{split}
\end{equation}
By \eqref{equ:Q-alpha-pm}, for any component $\widetilde{J}$ of $Q_\alpha^{-m}(J_\alpha^+)\setminus Q_\alpha^{-(m-1)}(J_\alpha^+)$ with $m\geqslant 1$, there exists a component $\widetilde{A}$ of $Q_\alpha^{-m}(A_\alpha^+)\setminus Q_\alpha^{-(m-1)}(A_\alpha^+)$ such that $\widetilde{J}\subset \widetilde{A}\cup (\widetilde{A})^\Int$ and $Q_\alpha^{\circ m}:\widetilde{A}\cup (\widetilde{A})^\Int\to A_\alpha^+\cup (A_\alpha^+)^\Int$ is conformal. The case for $J_\alpha^-$ is the same. Hence Part (b) holds since $J_\alpha^\infty$ and $J_\alpha^0$ are quasiconformally homeomorphic to $J(P_\alpha)$ and $J(P_{-\alpha})$ respectively.

 \begin{figure}[tpb]
  \setlength{\unitlength}{1mm}
  \setlength{\fboxsep}{0pt}
  \centering
  \fbox{\includegraphics[width=0.9\textwidth]{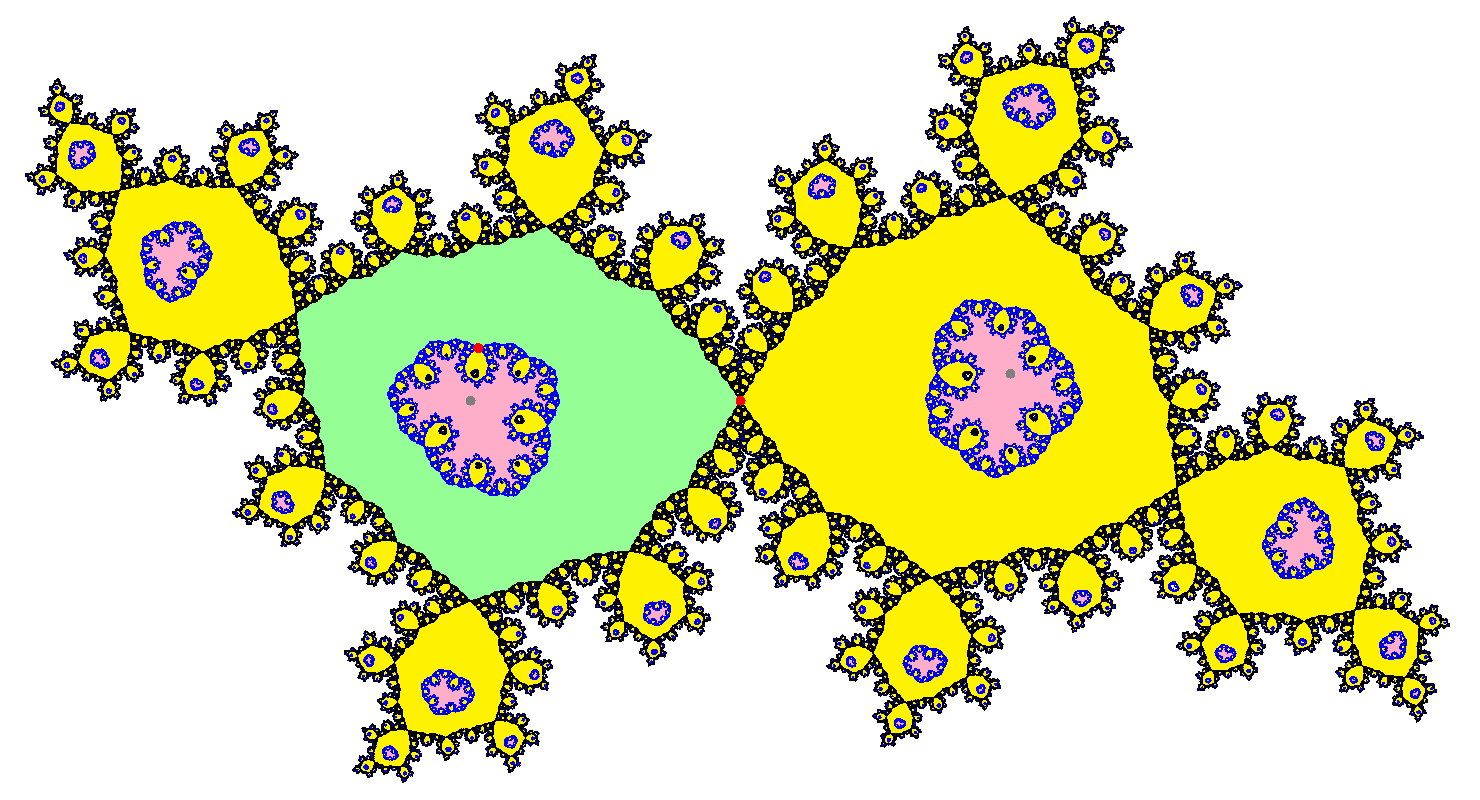}}
  \put(-75.5,33.3){$A_\alpha$}
  \put(-57.5,33.3){$A_\alpha^+$}
  \put(-67.5,33.3){$1$}
  \put(-87.3,40.5){$c$}
  \put(-42.5,34){$a$}
  \put(-87.5,31.1){\small{$0$}} \vskip0.3cm %
  \fbox{\includegraphics[width=0.435\textwidth]{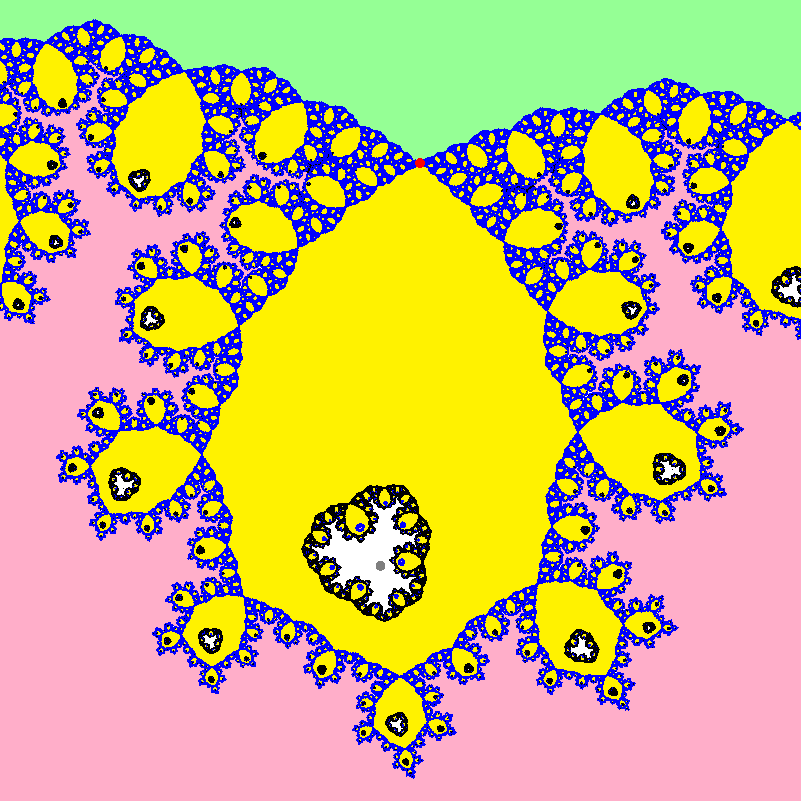}}\quad
  \fbox{\includegraphics[width=0.435\textwidth]{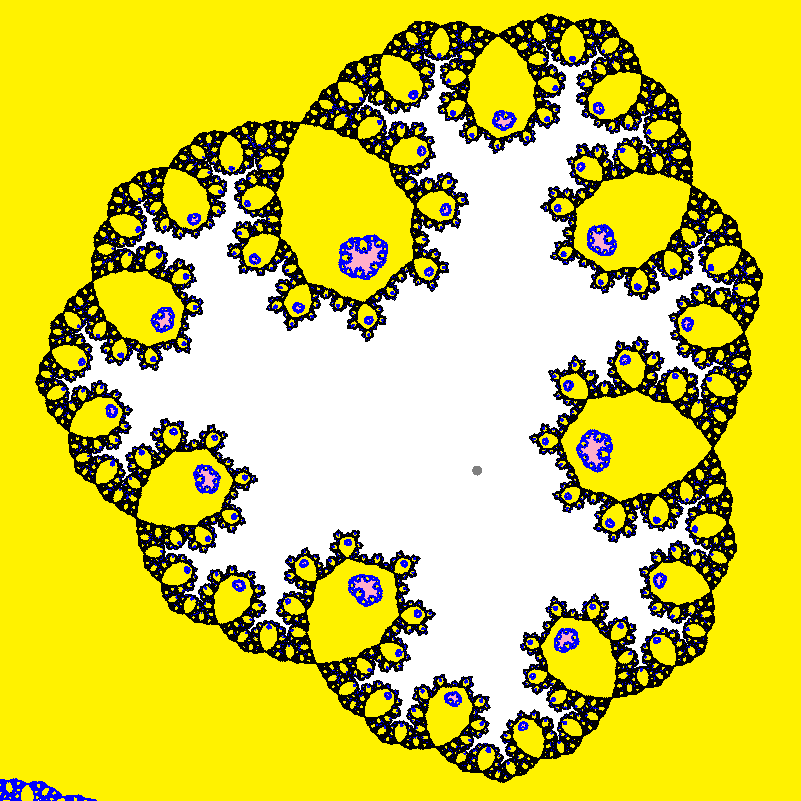}}
  \put(-90,57){$A_\alpha$}
  \put(-99,34){$A_\alpha^-$}
  \put(-96,51){$c$}
  \put(-101.7,18.6){\tiny{$1/b$}}
  \put(-24,24.6){$\frac{1}{b}$}
  \put(-57,57){$A_\alpha^-$}
  \put(-36,28){$B_\alpha^-$}
  \caption{A cubic map $Q_\alpha$ with a \textit{non-symmetric} Herman ring $A_\alpha$ (colored green) and its successive zooms, where $a=2+\frac{\ii}{10}$ and $u\approx -3.98404183+3.28819628\ii$ are chosen such that the rotation number is $\alpha=(\sqrt{5}-1)/2$. The Julia components $J_\alpha^\infty$ and $J_\alpha^0$ with preimages, the preimages of $A_\alpha$ and two super-attracting basins $\MB_\alpha(0)$, $\MB_\alpha(\infty)$ are colored black, blue, yellow, pink and white respectively. Some special points (critical points $1$ and $c=\frac{a(a-2)}{2a-3}$, zeros $0$ and $a$, and the pole $\frac{1}{b}=\frac{a-2}{2a-3}$) are also marked.}
  \label{Fig-Herman-no-symm}
\end{figure}
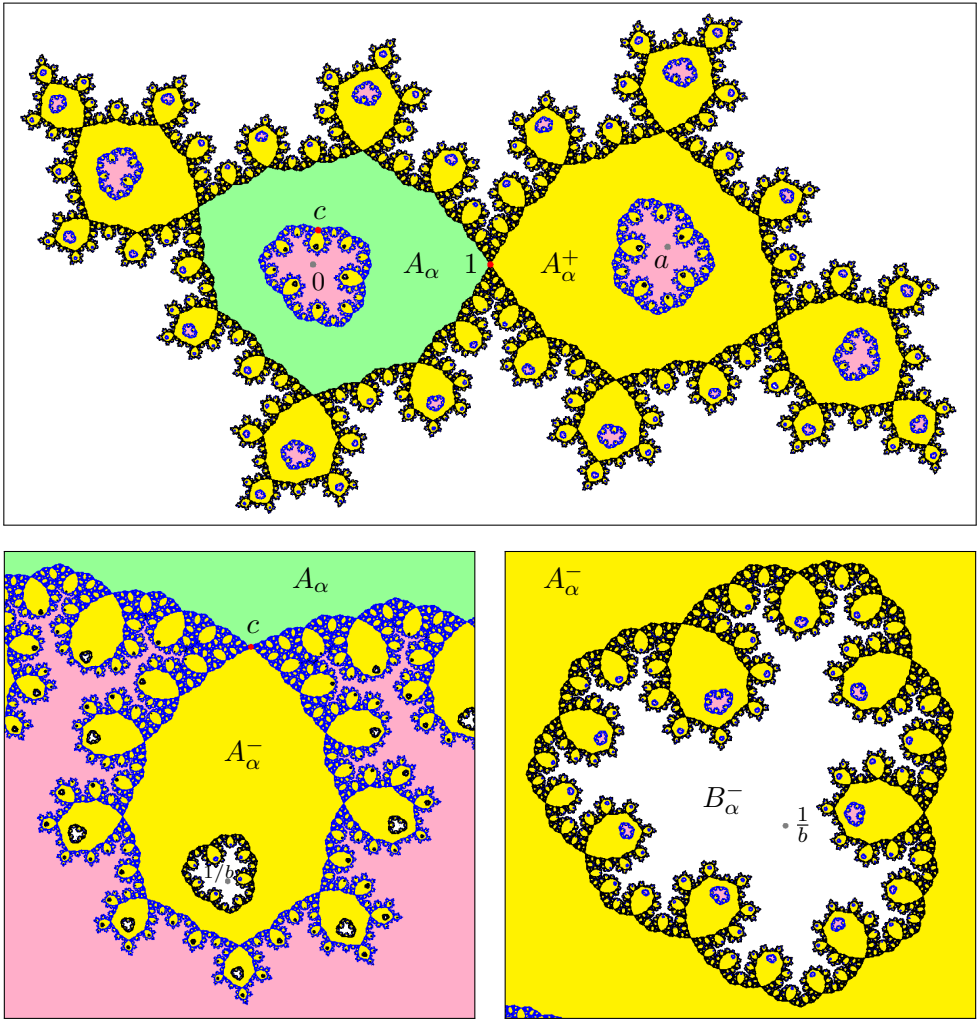

\medskip
(c) We prove this part by using a criterion of McMullen \cite[\S 2.8]{McM94b}. The main difference is that in McMullen's case the number of the annuli in each level is finite while in our case the number of annuli forms a countable set. Let $A_\alpha^+$ and $A_\alpha^-$ be the preimages of $A_\alpha$ contained in $A_\alpha^\Ext$ and $A_\alpha^\Int$ respectively (see \eqref{equ:A-alpha-pm}). If there exist domains $\Omega^+$ and $\Omega^-$ such that $Q_\alpha^{\circ m}: \Omega^+ \to A_\alpha\cup A_\alpha^{\Int}$ and $Q_\alpha^{\circ m}: \Omega^- \to A_\alpha\cup A_\alpha^{\Ext}$ are conformal for some $m\geqslant 1$, by \eqref{equ:Q-alpha-pm} and \eqref{equ:B-alpha-pm}, we have
\begin{equation}\label{equ:path}
Q_\alpha^{\circ (m-1)}(\Omega^+)=A_\alpha^+\cup (A_\alpha^+)^{\Int} \text{\quad and\quad}
Q_\alpha^{\circ (m-1)}(\Omega^-)=A_\alpha^-\cup (A_\alpha^-)^{\Int}.
\end{equation}

\textbf{Step 1} (Definition of $\ME_1$ and $\ME_2$). Define
\begin{equation}\label{equ:E-1}
\ME_1:= A_\alpha\cup
\left\{
\, A\,
\left|
\begin{array}{l}
A \text{ is a component of } Q_\alpha^{-m}(A_\alpha) \text{ for some } m\geqslant 1 \\
\text{and } Q_\alpha^{\circ k}(A)\subset A_\alpha^{\Ext} \text{ for all }0\leqslant k\leqslant m-1
\end{array}
\right.
\right\}.
\end{equation}
Then $\ME_1$ consists of countably many pairwise disjoint non-nested annuli with the same modulus as $A_\alpha$, and for each $A_1\in\ME_1\setminus\{A_\alpha\}$, there exists a unique $m=m(A_1)\geqslant 1$ such that
\begin{equation}\label{equ:A-1}
Q_\alpha^{\circ m}: A_1\cup A_1^{\Int} \to A_\alpha\cup A_\alpha^{\Int}
\end{equation}
is conformal and $Q_\alpha^{\circ m}(A_1^{\Int})=A_\alpha^{\Int}$. Moreover,
\begin{equation}\label{equ:A-alpha-ext}
A_\alpha^{\Ext}=B_\alpha^\infty\cup J_\alpha^\infty\cup \big\{A_1\cup A_1^{\Int}: A_1\in \ME_1\setminus\{A_\alpha\}\big\}.
\end{equation}

Define
\begin{equation}\label{equ:E-2-A-alpha}
\ME_2^{A_\alpha}:=
\left\{
\, A\,
\left|
\begin{array}{l}
A \text{ is a component of } Q_\alpha^{-m}(A_\alpha) \text{ for some } m\geqslant 1 \\
\text{and } Q_\alpha^{\circ k}(A)\subset A_\alpha^{\Int} \text{ for all }0\leqslant k\leqslant m-1
\end{array}
\right.
\right\}.
\end{equation}
Then $\ME_2^{A_\alpha}$ consists of countably many pairwise disjoint non-nested annuli with the same modulus as $A_\alpha$, and for each $A_2\in\ME_2^{A_\alpha}$, there exists a unique $m=m(A_2)\geqslant 1$ such that
\begin{equation}\label{equ:A-2}
Q_\alpha^{\circ m}: A_2\cup A_2^{\Int} \to A_\alpha\cup A_\alpha^{\Ext}
\end{equation}
is conformal and $Q_\alpha^{\circ m}(A_2^{\Int})=A_\alpha^{\Ext}$. Moreover,
\begin{equation}\label{equ:A-alpha-int}
A_\alpha^{\Int}=B_\alpha^0\cup J_\alpha^0\cup \big\{A_2\cup A_2^{\Int}: A_2\in \ME_2^{A_\alpha}\big\}.
\end{equation}
For any $A_1\in\ME_1\setminus\{A_\alpha\}$ (hence satisfying \eqref{equ:A-1} with $m=m(A_1)$), we define
\begin{equation}
\ME_2^{A_1}:=\big\{A\,|\,A\subset A_1^\Int \text{ and } Q_\alpha^{\circ m}(A)\in \ME_2^{A_\alpha}\big\}.
\end{equation}
Denote
\begin{equation}
\ME_2:=\bigcup_{A_1\in\ME_1} \ME_2^{A_1}.
\end{equation}
Then $\ME_2$ consists of countably many pairwise disjoint non-nested annuli with the same modulus as $A_\alpha$ and any $A_2\in\ME_2$ is nested inside a unique $A_1\in\ME_1$.
See Figure \ref{Fig-Herman-no-symm}.

\medskip
\textbf{Step 2} (The induction).
Suppose that $\ME_1$, $\ME_2$, $\cdots$, $\ME_n$ have been defined for some $n\geqslant 2$, such that it is a sequence of disjoint open sets in $\C$ satisfying
\begin{itemize}
\item $\ME_k$ consists of countably many pairwise disjoint non-nested annuli with the same modulus as $A_\alpha$, where $1\leqslant k\leqslant n$;
\item Each $A_k\in\ME_k$ is nested inside a unique $A_{k-1}\in\ME_{k-1}$, where $2\leqslant k\leqslant n$;
\item If $1\leqslant k\leqslant n$ is odd, then for any $A_k\in\ME_k\setminus\{A_\alpha\}$ (Note that $A_\alpha\in\ME_1$ and $A_\alpha\not\in\ME_k$ for any $2\leqslant k\leqslant n$), there exists a unique $m=m(A_k)\geqslant 1$ such that $Q_\alpha^{\circ m}: A_k\cup A_k^{\Int} \to A_\alpha\cup A_\alpha^{\Int}$ is conformal; and
\item If $2\leqslant k\leqslant n$ is even, then for any $A_k\in\ME_k$, there exists a unique $m=m(A_k)\geqslant 1$ such that $Q_\alpha^{\circ m}: A_k\cup A_k^{\Int} \to A_\alpha\cup A_\alpha^{\Ext}$ is conformal.
\end{itemize}
Then for odd $1\leqslant k\leqslant n$ and each $A_k\in\ME_k\setminus\{A_\alpha\}$, we have
\begin{equation}
A_k^{\Int}=B^{A_k}\cup J^{A_k}\cup \big\{A\cup A^{\Int}: A\subset A_k^\Int \text{ and } Q_\alpha^{\circ m}(A)\in \ME_2^{A_\alpha}\big\},
\end{equation}
where $B^{A_k}$ and $J^{A_k}$ are the unique components of $Q_\alpha^{-m}(B_\alpha^0)$ and $Q_\alpha^{-m}(J_\alpha^0)$ contained in $A_k^{\Int}$ respectively with $m=m(A_k)$.
For even $2\leqslant k\leqslant n$ and each $A_k\in\ME_k$, we have
\begin{equation}
A_k^{\Int}=B^{A_k}\cup J^{A_k}\cup \big\{A\cup A^{\Int}: A\subset A_k^\Int \text{ and } Q_\alpha^{\circ m}(A)\in \ME_1\setminus\{A_\alpha\}\big\},
\end{equation}
where $B^{A_k}$ and $J^{A_k}$ are the unique components of $Q_\alpha^{-m}(B_\alpha^\infty)$ and $Q_\alpha^{-m}(J_\alpha^\infty)$ contained in $A_k^{\Int}$ respectively with $m=m(A_k)$.

If $n$ is odd, for $A_n\in\ME_n$, we define
\begin{equation}
\ME_{n+1}^{A_n}:=\big\{A\,|\,A\subset A_n^\Int \text{ and } Q_\alpha^{\circ m}(A)\in \ME_2^{A_\alpha}\big\},
\end{equation}
where $m=m(A_n)\geqslant 1$. For any $A_{n+1}\in\ME_{n+1}^{A_n}$, by \eqref{equ:A-2} there exists $m'=m'(A_{n+1})\geqslant 1$ such that $Q_\alpha^{\circ m'}: A_{n+1}\cup A_{n+1}^{\Int} \to A_\alpha\cup A_\alpha^{\Ext}$ is conformal.
If $n$ is even, for $A_n\in\ME_n$, we define
\begin{equation}
\ME_{n+1}^{A_n}:=\big\{A\,|\,A\subset A_n^\Int \text{ and } Q_\alpha^{\circ m}(A)\in \ME_1\setminus\{A_\alpha\}\big\},
\end{equation}
where $m=m(A_n)\geqslant 1$. For any $A_{n+1}\in\ME_{n+1}^{A_n}$,  by \eqref{equ:A-1} there exists $m'=m'(A_{n+1})\geqslant 1$ such that $Q_\alpha^{\circ m'}: A_{n+1}\cup A_{n+1}^{\Int} \to A_\alpha\cup A_\alpha^{\Int}$ is conformal.
In both cases, we define
\begin{equation}
\ME_{n+1}:=\bigcup_{A_n\in\ME_n} \ME_{n+1}^{A_n}.
\end{equation}
Then $\ME_{n+1}$ consists of countably many pairwise disjoint non-nested annuli with the same modulus as $A_\alpha$ and any $A_{n+1}\in\ME_{n+1}$ is nested inside a unique $A_n\in\ME_n$.
This finishes the induction.

We have a sequence of disjoint open sets $\ME_1$, $\ME_2$, $\cdots$, $\ME_n$ , $\cdots$ in $\C$ satisfying
\begin{itemize}
\item $\ME_n$ consists of countably many pairwise disjoint non-nested annuli with the same modulus as $A_\alpha$, for all $n\geqslant 1$;
\item Each $A_n\in\ME_n$ is nested inside a unique $A_{n-1}\in\ME_{n-1}$, for all $n\geqslant 2$;
\item If $n\geqslant 1$ is odd, then for any $A_n\in\ME_n\setminus\{A_\alpha\}$, there exists $m=m(A_n)\geqslant 1$ such that $Q_\alpha^{\circ m}: A_n\cup A_n^{\Int} \to A_\alpha\cup A_\alpha^{\Int}$ is conformal; and
\item If $n\geqslant 2$ is even, then for any $A_n\in\ME_n$, there exists $m=m(A_n)\geqslant 1$ such that $Q_\alpha^{\circ m}: A_n\cup A_n^{\Int} \to A_\alpha\cup A_\alpha^{\Ext}$ is conformal.
\end{itemize}

\medskip
\textbf{Step 3} (Applying McMullen's criterion).
For $n\geqslant 1$, denote
\begin{equation}
\MF_n:=\bigcup_{A_n\in\ME_n} A_n^{\Int} \text{\quad and\quad} \MF:=\bigcap_{n\geqslant 1}\MF_n.
\end{equation}
Obviously $\MF$ is non-empty.
For any $z\in\MF$, there exists a sequence of nested annuli $(A_n)_{n\geqslant 1}$ such that $z\in A_n^{\Int}$ and $A_n\in\ME_n$ for all $n\geqslant 1$.
Since $\Mod(A_n)=\Mod(A_\alpha)$ for any $n\geqslant 1$, it follows that $\bigcap_{n\geqslant 1}A_n^{\Int}=\{z\}$. Since each $A_n$ is a Fatou component, this implies that $z\in J(Q_\alpha)$ and $\MF$ is a totally disconnected subset of $J(Q_\alpha)$.
By the definition of $\MF_n$, $J(Q_\alpha)$ is the disjoint union of $\MJ_\alpha^1$ and $\MF$. Hence we have $\MF=\MJ_\alpha^2$.

It remains to prove that $\MF$ has area zero. This follows by a similar proof of \cite[Theorem 2.16]{McM94b}. Applying the area-modulus estimate inequality (see \cite[Lemma 2.17]{McM94b}), we have
\begin{equation}
\area(\MF_n)\leqslant \Big(\area(\MF_1)+\sum_{A_1\in\ME_1}\area(A_1)\Big)\sup_{\MA_n}\prod_{k=1}^n\frac{1}{1+4\pi\,\Mod(A_k)},
\end{equation}
where $\MA_n$ denotes the collection of all sequences of nested annuli $(A_k)_{1\leqslant k\leqslant n}$ with $A_k\in \ME_k$.
Since $\Mod(A_k)=\Mod(A_\alpha)$ for all $k$, we have $\area(\MF_n)\to 0$ as $n\to\infty$ and hence $\area(\MF)=\area(\MJ_\alpha^2)=0$.
\end{proof}

Note that $\MJ_\alpha^2$ is not a Cantor set since it is not compact. Indeed, $\MJ_\alpha^2$ is a completely invariant proper subset of the Julia set of $Q_\alpha$. Every point of $\MJ_\alpha^2$ is a \textit{buried component}, i.e., it is a Julia component which is disjoint from the boundary of any Fatou component of $Q_\alpha$. This phenomenon for cubic Blaschke products was first confirmed in \cite{Qia95} (see also \cite{Bea91a}).

\medskip
An irrational number $\alpha\in\R\setminus\Q$ is called \textit{bounded type} if $\sup_{n\geqslant 1}\{a_n\}<+\infty$, where $[a_0;a_1,a_2,\cdots,a_n,\cdots]$ is the continued fraction expansion of $\alpha$. There are fruitful results on the topology and geometry of the Julia set of $P_\alpha(z)=e^{2\pi\ii\alpha}z+z^2$ when $\alpha$ is of bounded type. Douady and Herman proved that $\partial\Delta_\alpha$ is a quasi-circle passing through the unique critical point $\omega_\alpha=-e^{2\pi\ii\alpha}/2$ based on a cubic Blaschke model (see \cite{Dou87}, \cite{Her87}). Petersen proved that $J(P_\alpha)$ is locally connected and has area zero \cite{Pet96} (see also \cite{Yam99}).
McMullen improved the area result and proved that $J(P_\alpha)$ has Hausdorff dimension strictly less than two \cite{McM98b}.

Note that the set of bounded type irrationals occupies zero Lebesgue measure among all irrational numbers. In 2004, Petersen and Zakeri introduced the following type of irrational numbers \cite{PZ04}:
\begin{equation}
\PZ:=\big\{\alpha\in\R\setminus\Q: \log a_n= \MO(\sqrt{n}) \text{ as }n\to\infty\big\}.
\end{equation}
They proved the following remarkable result.

\begin{thm}[{Petersen-Zakeri}]\label{thm-PZ}
The set $\PZ$ occupies full measure in $\R\setminus\Q$, and for all $\alpha\in\PZ$, $J(P_\alpha)$ is locally connected and has area zero. In particular, $\partial\Delta_\alpha$ is a Jordan curve passing through the critical point $\omega_\alpha$.
\end{thm}

In the following we refer all irrational numbers in $\PZ$ the \textit{Petersen-Zakeri type}.

\begin{cor}\label{cor:zero-area}
For any Brjuno number $\alpha$, we have $\area(J(Q_\alpha))=0$ if and only if $\area(J(P_\alpha))=0$.
If $\alpha\in\PZ$, then each connected component of the boundary of the Herman ring of $Q_\alpha$ is a Jordan curve passing through exactly one critical point and $\area(J(Q_\alpha))=0$.
 \end{cor}

\begin{proof}
The ``if" part of the first statement and the second statement follow immediately from Lemma \ref{lem:zero-area} and Theorem \ref{thm-PZ}.
To prove the ``only if" part of the first statement, we use a standard Herman-to-Siegel quasiconformal surgery (which is the inverse process of Siegel-to-Herman surgery, see \cite{Shi87} or \cite[\S 5]{Shi06}).
Let $\gamma$ be an invariant analytic curve in the Herman ring $A_\alpha$ of $Q_\alpha$.
Then there exists an real-analytic map $\xi_\alpha:\T\to\gamma$ which conjugates the irrational rotation $\zeta\mapsto e^{2\pi\ii\alpha} \zeta$ to $Q_\alpha: \gamma\to \gamma$.
Hence $\xi_\alpha^{-1}:\gamma\to\T$ can be extended continuously to $\Theta_\alpha:\overline{\gamma^{\Int}}\to\overline{\D}$ such that $\Theta_\alpha: \gamma^{\Int}\to \D$ is quasiconformal (see \cite[Proposition 2.30(a)]{BF14a}). Define
\begin{equation}
G_\alpha(z):=
\left\{
\begin{array}{ll}
Q_{\alpha}(z)  &~~~~~~~\text{if}~z\in \overline{\gamma^{\Ext}}, \\
\Theta_\alpha^{-1}\big(e^{2\pi\ii\alpha}\Theta_\alpha(z)\big) &~~~~~~\text{if}~z\in\gamma^{\Int}.
\end{array}
\right.
\end{equation}
Then $G_\alpha:\EC\to\EC$ is a quasi-regular map of degree two.
Let $\sigma:=\Theta^*(\sigma_0)$, where $\sigma_0$ is the standard complex structure in $\D$.
We pull back $\sigma$ from $\gamma^{\Int}$ to $Q_\alpha^{-k}\big(\gamma^{\Int}\big)$ by defining $\sigma:=(Q_\alpha^{\circ k})^*(\sigma)$ for all $k\geqslant 1$. In the rest place we define $\sigma:=\sigma_0$. Since $\sigma_0$ is invariant under $\zeta\mapsto e^{2\pi\ii\alpha} \zeta$, it follows that $\sigma$ is invariant under $G_\alpha:\EC\to\EC$.
Then there exists a quasiconformal mapping $\Upsilon_\alpha:\EC\to\EC$ integrating $\sigma$ such that $\Upsilon_\alpha\circ G_\alpha\circ\Upsilon_\alpha^{-1}$ is the quadratic polynomial $P_{\alpha}$ and $J(P_\alpha)=\Upsilon_\alpha(J_\alpha^\infty)$, where $J_\alpha^\infty$ is the immediate super-attracting basin of $Q_\alpha$ containing $\infty$. If $\area(J(Q_\alpha))=0$, then $\area(J_\alpha^\infty)=0$ and hence $\area(J(P_\alpha))=0$ since $\Upsilon_\alpha$ is absolutely continuous with respect to $2$-dimensional Lebesgue measure.
 \end{proof}

\begin{rmk}
In fact, in subsequent proofs, bounded type rotation numbers are sufficient for our purpose. We consider the Petersen-Zakeri type so that some results can be applied to more general cases.
\end{rmk}

\subsection{Rigidity of the surgery} \label{subsec:rigidity}

For an annulus $A$ in $\C$, we use $\partial_+A$ and $\partial_-A$ to denote the exterior and interior boundary components of $A$ in $A^{\Ext}$ and $A^{\Int}$ respectively.
Assume that $\partial_+A$ and $\partial_-A$ are Jordan curves containing two marked points $c_+$ and $c_-$ respectively. Then there exists a unique homeomorphism
\begin{equation}
\varphi: \overline{A}\to\overline{\A}_{r,1}=\{z\in\C: r\leqslant |z|\leqslant 1\}
\end{equation}
such that $\varphi:A\to\A_{r,1}$ is conformal and $\varphi(c_+)=1$, where $r\in(0,1)$. We call
\begin{equation}
\theta:=\arg\varphi(c_-)\in \R/(2\pi\Z)
\end{equation}
the \textit{conformal angle} between $c_+$ and $c_-$ (with respect to $A$). The real-analytic curve $\varphi^{-1}(\T_{\sqrt{r}})$ is called the \textit{core curve} of $A$.

Recall that $\psi_{\alpha,+}(\zeta)=\phi_\alpha\big(r e^{\ii\theta_1}\zeta\big)$ and $\psi_{\alpha,-}(\zeta)= \eta_r\circ\phi_{-\alpha}\Big(\tfrac{r}{e^{\ii\theta_2}\zeta}\Big)$ are the conformal maps defined in \eqref{equ:psi-alpha-pm}, where $r\in(0,r_\alpha)$ and $r_\alpha>0$ is the conformal radius of the Siegel disk $\Delta_\alpha$ of $P_\alpha$. The numbers $\theta_1,\theta_2\in\R$ determine the conformal angle as following.

\begin{lem}\label{lem:conformal-angle}
Let $\alpha\in\PZ$. Then the conformal angle between the two critical points $1$ and $c$ with respect to the Herman ring $A_\alpha$ of $Q_\alpha$ is $\theta_1-\theta_2$.
\end{lem}

\begin{proof}
By Lemma \ref{lem:surgery-S-to-H} and Corollary \ref{cor:zero-area}, $\partial A_\alpha$ consists of two Jordan curves with $1\in\partial_+A_\alpha$ and $c\in\partial_-A_\alpha$.
Let $\phi_\alpha:\D_{r_\alpha}\to\Delta_\alpha$ be the conformal map satisfying $\phi_\alpha(0)=0$ and $\phi_\alpha'(0)=1$. Then $\phi_\alpha$ can be extended to a homeomorphism $\phi_\alpha:\overline{\D}_{r_\alpha}\to\overline{\Delta}_\alpha$ and there exists a unique $\theta_\alpha\in\R/(2\pi\Z)$ such that $\phi_\alpha(r_\alpha e^{\ii\theta_\alpha} )=\omega_\alpha$, where $\omega_\alpha=-e^{2\pi\ii\alpha}/2\in\partial\Delta_\alpha$ is the critical point of $P_\alpha$. Thus we have
\begin{equation}\label{equ:crit-1-position}
\psi_{\alpha,+}\big(\tfrac{r_\alpha}{r}e^{\ii(\theta_\alpha-\theta_1)}\big)=\phi_\alpha(r_\alpha e^{\ii\theta_\alpha})=\omega_\alpha.
\end{equation}

By \eqref{equ:phi-minus-alpha}, we have $\phi_{-\alpha}(r_\alpha e^{-\ii\theta_\alpha})=\omega_{-\alpha}$. Hence
\begin{equation}\label{equ:crit-c-position}
\eta_r^{-1}\circ\psi_{\alpha,-}\big(\tfrac{r}{r_\alpha}e^{\ii(\theta_\alpha-\theta_2)}\big)=\phi_{-\alpha}(r_\alpha e^{-\ii\theta_\alpha})=\omega_{-\alpha}.
\end{equation}
Let $\Phi_\alpha:\EC\to\EC$ be the quasiconformal mapping in Lemma \ref{lem:surgery-S-to-H} satisfying $\Phi_\alpha(\omega_\alpha)=1$ and $\Phi_\alpha(\eta_r(\omega_{-\alpha}))=c$ (see the last paragraph in the proof). By \eqref{equ:Psi}, \eqref{equ:Psi-Phi}, \eqref{equ:crit-1-position} and \eqref{equ:crit-c-position}, the conformal angle between the two critical points $1$ and $c$ with respect to the Herman ring $A_\alpha$ of $Q_\alpha$ is $\theta_1 -\theta_2$.
\end{proof}

\begin{lem}\label{lem:rigidity}
Let $\alpha\in\PZ$, $0<r<r_\alpha$ and $\theta\in[0,2\pi)$. Then there exist unique parameters $a=a(\alpha,r,\theta)\in\C\setminus\big\{0,1,\frac{3}{2},2,3\big\}$ and $u=u(\alpha,r,\theta)\in\C\setminus\{0\}$ such that
\begin{equation}\label{equ:Q-a-u}
Q_{a,u}(z):=u z^2\frac{z-a}{1-\tfrac{2a-3}{a-2} z}
\end{equation}
has a fixed Herman ring $A$ in $\C$ satisfying the following properties:
\begin{itemize}
\item $Q_{a,u}:A\to A$ has rotation number $\alpha$;
\item $\Mod(A)=\frac{1}{\pi}\log\frac{r_\alpha}{r}$; and
\item $\partial_+A$ and $\partial_-A$ are Jordan curves passing through the critical points $1$ and $c=\frac{a(a-2)}{2a-3}$ respectively, and the conformal angle between $1$ and $c$ is $\theta$.
\end{itemize}
\end{lem}

\begin{proof}
The parameters $a$ and $u$ satisfying the properties in this lemma exist. Indeed, for given $\alpha\in\PZ$, $0<r<r_\alpha$ and $\theta\in[0,2\pi)$, by the Siegel-to-Herman surgery in \S\ref{subsec:S-to-H}, Lemma \ref{lem:surgery-S-to-H} guarantees that $A$ has rotation number $\alpha$ and modulus $\frac{1}{\pi}\log\frac{r_\alpha}{r}$. Corollary \ref{cor:zero-area} and Lemma \ref{lem:conformal-angle} guarantee that each component of $\partial A$ is a Jordan curve passing through a critical point, and the conformal angle between $1$ and $c$ can be adjusted to be $\theta$ (i.e., set $\theta=\theta_1-\theta_2$). It remains to prove the uniqueness of $a$ and $u$.

Suppose that $Q_{a_1,u_1}$ and $Q_{a_2,u_2}$ satisfy all the properties of this lemma. In the following we prove that $Q_{a_1,u_1}:\EC\to\EC$ is conjugate to $Q_{a_2,u_2}:\EC\to\EC$ by a conformal map. The strategy is to use the dynamics to lift quasiconformal mappings which are conformal in the periodic Fatou components. One can obtain a sequence of quasiconformal mappings which are ``more and more" conformal and the limit is quasiconformal in $\EC$ and conformal in the Fatou set. Since the Julia sets have area zero, then the limit is conformal everywhere. In fact, such an idea is very standard in the proof of rigidity.

\medskip
Let us give the details of the proof. For simplicity, we use the subscript ``$i$" to denote the objects corresponding to $Q_i:=Q_{a_i,u_i}$ as in \S\ref{subsec:zero-area} for $i=1,2$. For example, the Herman ring for $Q_{i}$ is denoted by $A_{i}$ whose boundary components $\partial_+A_i$ and $\partial_-A_i$ contain the critical points $1$ and $c_i$ respectively, the immediate super-attracting basin of $\infty$ (resp. $0$) of $Q_{i}$ is denoted by $B_{i}^\infty$ (resp. $B_{i}^0$) etc.
Let $\gamma_{i}$ be the core curve of $A_{i}$. Then $\gamma_{i}$ is a real-analytic Jordan curve in $A_{i}$ satisfying\footnote{Note that $\gamma_{i}^{\Int}$ is an open disk while $A_{i}^{\Int}$ is a closed connected set.}
\begin{equation}
\Mod\big(A_{i}\setminus\overline{\gamma_{i}^{\Int}}\big)=\Mod\big(\gamma_{i}^{\Int}\setminus A_{i}^{\Int}\big)=\frac{1}{2\pi}\log\frac{r_\alpha}{r}.
\end{equation}
Since $\Mod(A_{1})=\Mod(A_{2})$, we use $\gamma_\alpha$ to denote the invariant real-analytic Jordan curve in the Siegel disk $\Delta_\alpha$ of $P_\alpha$ such that
\begin{equation}
\Mod\big(A_{1}\setminus\overline{\gamma_{1}^{\Int}}\big)=\Mod\big(A_{2}\setminus\overline{\gamma_{2}^{\Int}}\big)
=\Mod\big(\Delta_\alpha\setminus\overline{\gamma_{\alpha}^{\Int}}\big)=\frac{1}{2\pi}\log\frac{r_\alpha}{r}.
\end{equation}
Let $\widehat{B}_\alpha(\infty)$ be the super-attracting basin of $\infty$ of $P_\alpha$.

\medskip
\textbf{Step 1} (Definition of $\varphi_0$).
By a standard Herman-to-Siegel quasiconformal surgery (see the proof of Corollary \ref{cor:zero-area}), there exist two homeomorphisms $\psi_i^+:\overline{\gamma_{i}^{\Ext}}\to\overline{\gamma_\alpha^{\Ext}}$, where $i=1,2$, such that
\begin{itemize}
\item $\psi_i^+:\gamma_{i}^{\Ext}\to\gamma_\alpha^{\Ext}$ is quasiconformal, and the restriction $\psi_i^+:\Omega_{i}^+\to \Xi_{\alpha}^+$ is conformal, where
$\Omega_{i}^+:=B_{i}^\infty\cup\big(A_{i}\setminus\overline{\gamma_{i}^{\Int}}\big)$ and
$\Xi_{\alpha}^+:=\widehat{B}_\alpha(\infty)\cup \big(\Delta_\alpha\setminus\overline{\gamma_{\alpha}^{\Int}}\big)$;
\item $\psi_i^+:\overline{\Omega_{i}^+}\to \overline{\Xi_{\alpha}^+}$ is a conjugacy between $Q_{i}:\overline{\Omega_{i}^+}\to \overline{\Omega_{i}^+}$ and $P_\alpha:\overline{\Xi_{\alpha}^+}\to \overline{\Xi_{\alpha}^+}$.
\end{itemize}
Hence $\psi_i^+$ maps the critical point $1$ of $Q_{i}$ to the critical point $\omega_\alpha$ of $P_\alpha$.
Then \begin{equation}\label{equ:varphi-0-plus}
\varphi_0^+:=(\psi_2^+)^{-1}\circ\psi_1^+: \overline{\gamma_{1}^{\Ext}}\to \overline{\gamma_{2}^{\Ext}}
\end{equation}
is a homemorphism satisfying
\begin{itemize}
\item $\varphi_0^+: \gamma_{1}^{\Ext}\to \gamma_{2}^{\Ext}$ is quasiconformal, and the restriction $\varphi_0^+:\Omega_{1}^+\to \Omega_{2}^+$ is conformal;
\item $\varphi_0^+:\overline{\Omega_{1}^+}\to \overline{\Omega_{2}^+}$ is a conjugacy between $Q_{1}:\overline{\Omega_{1}^+}\to \overline{\Omega_{1}^+}$ and $Q_{2}:\overline{\Omega_{2}^+}\to \overline{\Omega_{2}^+}$, $\varphi_0^+(1)=1$ and $\varphi_0^+(\infty)=\infty$.
\end{itemize}

Similar to \eqref{equ:varphi-0-plus}, there exists a homeomorphism
$\varphi_0^-: \overline{\gamma_{1}^{\Int}}\to \overline{\gamma_{2}^{\Int}}$
such that
\begin{itemize}
\item $\varphi_0^-: \gamma_{1}^{\Int}\to \gamma_{2}^{\Int}$ is quasiconformal, and the restriction $\varphi_0^-:\Omega_{1}^-\to \Omega_{2}^-$ is conformal, where $\Omega_{i}^-:=B_{i}^0\cup\big(\gamma_{i}^{\Int}\setminus A_{i}^{\Int}\big)$;
\item $\varphi_0^-:\overline{\Omega_{1}^-}\to \overline{\Omega_{2}^-}$ is a conjugacy between $Q_{1}:\overline{\Omega_{1}^-}\to \overline{\Omega_{1}^-}$ and $Q_{2}:\overline{\Omega_{2}^-}\to \overline{\Omega_{2}^-}$, $\varphi_0^-(c_1)=c_2$ and $\varphi_0^-(0)=0$.
\end{itemize}

We claim that $\varphi_0^+|_{\gamma_{1}}=\varphi_0^-|_{\gamma_{1}}$. Indeed, since $\Mod(A_{1})=\Mod(A_{2})$, there exist $0<\rho<1$ and two homeomorphisms $\tau_i: \overline{A}_i\to\overline{\A}_{\rho,1}$, where $i=1,2$, such that $\tau_i: A_i\to \A_{\rho,1}$ is conformal and $\tau_1(1)=\tau_2(1)=1$. Since the conformal angle between $1$ and $c_1$ with respect to $A_{1}$ is equal to that between $1$ and $c_2$ with respect to $A_{2}$, we have $\tau_1(c_1)=\tau_2(c_2)\in\T_{\rho}$.
Note that $\gamma_i=\tau_i^{-1}(\T_{\sqrt{\rho}})$ for $i=1,2$.
Then $\tau_2\circ\varphi_0^+\circ\tau_1^{-1}:\overline{\A}_{\sqrt{\rho},1}\to \overline{\A}_{\sqrt{\rho},1}$ is a homeomorphism fixing $1$ and satisfying that $\tau_2\circ\varphi_0^+\circ\tau_1^{-1}:\A_{\sqrt{\rho},1}\to \A_{\sqrt{\rho},1}$ is conformal. Hence $\tau_2\circ\varphi_0^+\circ\tau_1^{-1}:\overline{\A}_{\sqrt{\rho},1}\to \overline{\A}_{\sqrt{\rho},1}$ is the identity and $\varphi_0^+=\tau_2^{-1}\circ\tau_1$ on $\overline{\gamma_1^{\Ext}}\cap A_1$.
By a similar argument, we conclude that $\tau_2\circ\varphi_0^-\circ\tau_1^{-1}:\overline{\A}_{\rho,\sqrt{\rho}}\to \overline{\A}_{\rho,\sqrt{\rho}}$ is the identity and $\varphi_0^-=\tau_2^{-1}\circ\tau_1$ on $\overline{\gamma_1^{\Int}}\cap A_1$.
This implies that $\varphi_0^+|_{\gamma_{1}}=\varphi_0^-|_{\gamma_{1}}$.
Hence the map
\begin{equation}
\varphi_0(z):=
\left\{
\begin{array}{ll}
\varphi_0^+(z)  &~~~~~~~\text{if}~ z\in \overline{\gamma_{1}^{\Ext}}, \\
\varphi_0^-(z)  &~~~~~~\text{if}~ z\in \gamma_{1}^{\Int}
\end{array}
\right.
\end{equation}
satisfies
\begin{itemize}
\item $\varphi_0: \EC\to\EC$ is quasiconformal, and the restriction $\varphi_0:\Omega_{1}\to \Omega_{2}$ is conformal, where
    \begin{equation}
    \Omega_{i}:=\Omega_{i}^+\cup \Omega_{i}^-=B_{i}^\infty\cup A_{i} \cup B_{i}^0;
    \end{equation}
\item $\varphi_0:\overline{\Omega}_{1}\to \overline{\Omega}_{2}$ is a conjugacy between $Q_{1}:\overline{\Omega}_{1}\to \overline{\Omega}_{1}$ and $Q_{2}:\overline{\Omega}_{2}\to \overline{\Omega}_{2}$, $\varphi_0(c_1)=c_2$, $\varphi_0(0)=0$, $\varphi_0(1)=1$ and $\varphi_0(\infty)=\infty$.
\end{itemize}

\medskip
\textbf{Step 2} (Definition of $\varphi_1$).
For $i=1,2$, we denote $\MU_i:=\MU_i^\infty\cup \MU_i^0$, where
\begin{equation}
\begin{split}
\MU_i^\infty:=&~\big\{U: U \text{ is a component of } \EC\setminus\overline{B_i^\infty\cup A_i\cup A_i^\Int}\big\} \text{\quad and} \\
\MU_i^0:=&~\big\{U: U \text{ is a component of } \EC\setminus\overline{B_i^0\cup A_i\cup A_i^\Ext}\big\}.
\end{split}
\end{equation}
Then $\MU_i$ consists of all components of $\EC\setminus\overline{\Omega}_i$. Note that they are not Fatou components. By Lemma \ref{lem:zero-area} and Theorem \ref{thm-PZ}, each component $U$ in $\MU_i$ is a Jordan disk.
Moreover, for each $U\in \MU_i^\infty$ (resp. $U\in \MU_i^0$), there exists a unique $n\geqslant 1$ such that
\begin{equation}\label{equ:U-A-i}
Q_i^{\circ n}: U\to A_i\cup A_i^\Int \quad (\text{resp. } Q_i^{\circ n}: U\to A_i\cup A_i^\Ext)
\end{equation}
is conformal, and $Q_i^{\circ k}(U)\in\MU_i^\infty$ (resp. $Q_i^{\circ k}(U)\in\MU_i^0$) for every $0\leqslant k\leqslant n-1$.

In \cite{PZ04} (see also \cite[\S 0]{Pet96}), Petersen and Zakeri assign a unique \emph{address} for each preimage of the Siegel disk $\Delta_\alpha$ under $P_\alpha$ as follows.
Let $V_0$ be the other component of $P_\alpha^{-1}(\Delta_\alpha)$ which is different from $\Delta_\alpha$. Denote the critical point $x_0:=\omega_\alpha=-e^{2\pi\ii\alpha}/2$. For each integer $s\geqslant 1$, there exist a unique $x_s\in\partial\Delta_\alpha$ and a Jordan domain $V_s$ attaching at $x_s$ such that $P_\alpha^{\circ s}:V_s\to V_0$ is conformal, $P_\alpha^{\circ s}(x_s)=x_0$ and $\overline{V}_s\cap\overline{\Delta}_\alpha=\{x_s\}$. For each $s\geqslant 1$, there exist a unique $y_s\in\partial V_0$ and a Jordan domain $W_s$ attaching at $y_s$ such that $P_\alpha:W_s\to V_{s-1}$ is conformal, $P_\alpha(y_s)=x_{s-1}$ and $\overline{W}_s\cap\overline{V}_0=\{y_s\}$.
Inductively, for $n\geqslant 2$ and $n$-tuple $(s_1,\ldots, s_n)$ of positive integers, there exist Jordan domains $V_{s_1,\ldots,s_n}$ and $W_{s_1,\ldots, s_n}$ attaching at the points $x_{s_1,\ldots,s_n}$ and $y_{s_1,\ldots,s_n}$ respectively, such that
$\overline{V}_{s_1,\ldots,s_n}\cap \overline{V}_{s_1,\ldots,s_{n-1}}=\{x_{s_1,\ldots,s_n}\}$ and $\overline{W}_{s_1,\ldots,s_n}\cap \overline{W}_{s_1,\ldots,s_{n-1}}=\{y_{s_1,\ldots,s_n}\}$;
\begin{equation}
P_\alpha(x_{s_1,\ldots,s_n})=P_\alpha(y_{s_1,\dots,s_n})=
\left\{
\begin{array}{ll}
x_{s_1-1,\ldots,s_n}  &~~~~~~~\text{if}~s_1\geqslant 2, \\
y_{s_2,\ldots,s_n} &~~~~~~\text{if}~s_1=1;
\end{array}
\right.
\end{equation}
and
\begin{equation}
P_\alpha(V_{s_1,\ldots,s_n})=P_\alpha(W_{s_1,\dots,s_n})=
\left\{
\begin{array}{ll}
V_{s_1-1,\ldots,s_n}  &~~~~~~~\text{if}~s_1\geqslant 2, \\
W_{s_2,\ldots,s_n} &~~~~~~\text{if}~s_1=1.
\end{array}
\right.
\end{equation}
If a preimage of the Siegel disk $\Delta_\alpha$ under $P_\alpha$ is $V_{s_1,\ldots,s_n}$ (resp. $W_{s_1,\ldots,s_n}$), we use $(s_1,\ldots,s_n)^+$ (resp. $(s_1,\ldots,s_n)^-$) to denote its address.
This assign a unique address for each preimage of the Siegel disk $\Delta_\alpha$ under $P_\alpha$.

\medskip
By the surgery construction (see the proof of Corollary \ref{cor:zero-area} and the properties of $\psi_i^+$ in Step 1), one can assign a unique address (i.e., $(s_1,\ldots,s_n)^{\pm}$) for each $U\in\MU_i^{\infty}$ by the one-to-one correspondence of the quasiconformal mapping $\psi_i^+$. Similarly, each $U\in\MU_i^0$ can be assigned a unique address (i.e., $(s_1,\ldots,s_n)^{\pm}$) which corresponds to a preimage of the Siegel disk $\Delta_{-\alpha}$ under $P_{-\alpha}$.

For any $U_1\in\MU_1^\infty$ (resp. $\MU_1^0$), there exists a unique $U_2\in\MU_2^\infty$ (resp. $\MU_2^0$) with the same address as $U_1$. We define
\begin{equation}
\varphi_1:=Q_{2}^{-1}\circ \varphi_0\circ Q_{1}:U_1\to U_2,
\end{equation}
where $Q_2^{-1}$ is the inverse of the conformal map $Q_2: U_2\to Q_2(U_2)$.
By Step 1, $\varphi_0:\overline{\Omega}_{1}\to \overline{\Omega}_{2}$ is a conjugacy between $Q_{1}$ and $Q_2$. Hence $\varphi_1:U_1\to U_2$ can be extended to a homeomorphism \begin{equation}
\varphi_1:\overline{U}_1\to \overline{U}_2, \text{\quad where \quad}\varphi_1|_{\partial U_1}=\varphi_0|_{\partial U_1}.
\end{equation}
For $z\in\overline{\Omega}_1=\overline{B_1^\infty\cup A_1 \cup B_1^0}$, we define $\varphi_1(z):=\varphi_0(z)$. Then $\varphi_1$ is defined on the whole Riemann sphere and
we have
\begin{equation}\label{equ:lift-0-1}
\varphi_0\circ Q_1(z)=Q_2\circ\varphi_1(z), \text{\quad for } z\in\EC.
\end{equation}

We claim that $\varphi_1:\EC\to\EC$ is an orientation-preserving homeomorphism. Indeed, by construction, $\varphi_1:\EC\to\EC$ is injective and surjective. It suffices to prove that $\varphi_1$ is continuous at any $z\in\EC$. This is obvious if $z$ lies in $\Omega_1$ or in any component of $\MU_1$. Hence we only need to consider $z\in\partial B_1^\infty\cup\partial B_1^0$.
Note that the restriction $\varphi_1=\varphi_0:\overline{B_1^\infty}\cup \overline{B_1^0}\to \overline{B_2^\infty}\cup \overline{B_2^0}$ is continuous.
By Theorem \ref{thm-PZ}, the spherical diameter of any sequence of different components $(U_1^{(n)})_{n\geqslant 0}$ in $\MU_1$ tends to zero as $n\to\infty$.
It follows that $\varphi_1:\EC\to\EC$ is continuous at $z\in\partial B_1^\infty\cup\partial B_1^0$. Hence $\varphi_1:\EC\to\EC$ is an orientation-preserving homeomorphism.
In particular, by \eqref{equ:lift-0-1}, $\varphi_1:\EC\to\EC$ is a quasiconformal mapping whose complex dilatation satisfies $K(\varphi_1)=K(\varphi_0)$ since $Q_1$ and $Q_2$ are analytic.

\medskip
\textbf{Step 3} (Inductive definition of $\varphi_n$).
Assume that for every $1\le k\le n$, we have a quasiconformal homeomorphism $\varphi_k :\EC \to \EC$ such that
\begin{itemize}
  \item $\varphi_k=\varphi_{k-1}$ in $Q_1^{-(k-1)}(\overline{\Omega}_1)$;
  \item $\varphi_k$ is conformal in every component of $Q_1^{-k}(\Omega_1)$; and
  \item $\varphi_{k-1}\circ Q_1=Q_2\circ\varphi_k$.
\end{itemize}

We define $\varphi_{n+1}$ similarly as $\varphi_1$. For any $U_1\in\MU_1^\infty$ (resp. $\MU_1^0$), there exists a unique $U_2\in\MU_2^\infty$ (resp. $\MU_2^0$) with the same address as $U_1$. Define
\begin{equation}
\varphi_{n+1}:=Q_{2}^{-1}\circ \varphi_n\circ Q_{1}:U_1\to U_2.
\end{equation}
Since $\varphi_n|_{\overline{\Omega}_1}=\varphi_0|_{\overline{\Omega}_1}$ and $\varphi_0:\overline{\Omega}_{1}\to \overline{\Omega}_{2}$ is a conjugacy between $Q_{1}$ and $Q_2$, it follows that $\varphi_{n+1}:U_1\to U_2$ can be extended to a homeomorphism
\begin{equation}
\varphi_{n+1}:\overline{U}_1\to \overline{U}_2, \text{\quad where \quad}\varphi_{n+1}|_{\partial U_1}=\varphi_{n}|_{\partial U_1}=\varphi_{0}|_{\partial U_1}.
\end{equation}
For $z\in\overline{\Omega}_1=\overline{B_1^\infty\cup A_1 \cup B_1^0}$, we define $\varphi_{n+1}(z):=\varphi_n(z)=\varphi_0(z)$. Then $\varphi_{n+1}$ is defined on the whole Riemann sphere and
\begin{itemize}
  \item $\varphi_{n+1}=\varphi_n$ in $Q_1^{-n}(\overline{\Omega}_1)$;
  \item $\varphi_{n+1}$ is conformal in every component of $Q_1^{-(n+1)}(\Omega_1)$; and
  \item $\varphi_n\circ Q_1=Q_2\circ\varphi_{n+1}$ and $\varphi_{n+1}:\EC\to\EC$ is a quasiconformal mapping satisfying $K(\varphi_{n+1})=K(\varphi_0)$.
\end{itemize}

\medskip
\textbf{Step 4} (The limit).
By induction, we have a sequence of quasiconformal homeomorphisms $\{\varphi_n: \EC\to\EC\}_{n\geqslant 0}$ such that each $\varphi_n$ is conformal in $Q_1^{-n}(\Omega_1)$ and its complex dilatation satisfies $K(\varphi_n)= K(\varphi_0)<+\infty$. Since $\varphi_n=\varphi_0$ in $\Omega_1$ for all $n\geqslant 1$, it follows that $\{\varphi_n\}_{n\geqslant 0}$ is locally uniformly bounded in $\C$ and hence it is a normal family.
We first take a convergent subsequence $\{\varphi_{n_k}\}_{k\geqslant 0}$ such that $\varphi_{n_k}\to\varphi$ as $k\to\infty$. Next we take a further convergent subsequence $\{\varphi_{n_{k_j}+1}\}_{j\geqslant 0}$  such that $\varphi_{n_{k_j}+1}\to\psi$ as $j\to\infty$. Then we obtain two limit quasiconformal mappings $\varphi$ and $\psi$, such that
\begin{itemize}
\item $\varphi$ and $\psi$ fix $0$, $1$ and $\infty$; and
\item $\varphi\circ Q_1=Q_2\circ\psi$ since $\varphi_n\circ Q_1=Q_2\circ\varphi_{n+1}$ for all $n\geqslant 0$.
\end{itemize}
Since $\varphi_{n+1}=\varphi_n$ in $Q_1^{-n}(\overline{\Omega}_1)$, it follows that $\varphi=\psi$ in $\bigcup_{n\geqslant 0}Q_1^{-n}(\Omega_1)$. By Lemma \ref{lem:zero-area}, we conclude that $\varphi=\psi$ in the Fatou set of $Q_1$. Hence $\varphi=\psi$ on $\EC$ because of the continuity. Then $Q_1$ and $Q_2$ are quasiconformally conjugate to each other on $\EC$ and they are conformally conjugate in the Fatou sets.

By Corollary \ref{cor:zero-area}, $J(Q_{1})$ has area zero. Hence $\varphi:\EC\to\EC$ is a quasiconformal mapping which is conformal almost everywhere. This implies that $\varphi:\EC\to\EC$ is conformal. Since $\varphi$ fixes $0$, $1$ and $\infty$, we conclude that $\varphi$ is the identity and hence $Q_{1}=Q_{2}$.
Thus we get $a_1=a_2$ and $u_1=u_2$.
\end{proof}

\begin{rmk}
In the proof of Lemma \ref{lem:rigidity}, the idea behinds Steps 2 to 4 is based on promoting a combinatorial equivalence to a global quasiconformal conjugacy, which is known as the ``Thurston algorithm". See \cite{DH93} and \cite[Appendix A]{McM98b}.
\end{rmk}

If $\alpha\in\PZ$, Lemma \ref{lem:rigidity} implies that $Q_\alpha:=Q_{\alpha,r,\theta}$ in Lemma \ref{lem:surgery-S-to-H} is \textit{independent} of the following:
\begin{itemize}
\item the interpolation $\psi_{\alpha,0}$ in \eqref{equ:psi-0}; and
\item the choice of $r'\in (r,r_\alpha)$ in Lemma \ref{lem:annulus}.
\end{itemize}
It depends \textit{only} on the parameters
\begin{equation}\label{equ:alpha-r-theta}
\alpha\in\PZ, \quad r\in(0,r_\alpha) \text{\quad and\quad} \theta\in\R/(2\pi\Z).
\end{equation}
In the remainder of the paper, by Lemma \ref{lem:conformal-angle} we take $\theta_1=0$ and $\theta_2=-\theta$ in \eqref{equ:psi-alpha-pm} and fix
\begin{equation}\label{equ:psi-pm-std}
\psi_{\alpha,+}(\zeta):=\phi_\alpha\big(r \zeta\big) \text{\quad and\quad}\psi_{\alpha,-}(\zeta):= \eta_r\circ\phi_{-\alpha}\Big(\tfrac{r e^{\ii\theta}}{\zeta}\Big).
\end{equation}
By Lemma \ref{lem:annulus}, the restriction of $\psi_{\alpha,+}$ in $\A_{1,r_\alpha/r}$ and the restriction of $\psi_{\alpha,-}$ in $\A_{r/r_\alpha,1}$ are conformal (note that $\A_{r/r_\alpha,r/r'}\subset\A_{r/r_\alpha,1}$ for any $r'\in(r,r_\alpha)$).

Under the assumption \eqref{equ:alpha-r-theta}, in the following we use both $Q_{a,u}$ and $Q_{\alpha,r,\theta}$ alternately, depending on the needs.

\section{Perturbation lemma and surgery's continuity}

In this section, we first state a perturbation lemma on the conformal radii of quadratic Siegel disks due to Avila-Buff-Ch\'eritat. After that we recall a result on the control of the loss of the area of quadratic filled-in Julia sets by Buff and Ch\'{e}ritat. Based on these results we prove the continuity of the Siegel-to-Herman surgery when the high type rotation numbers tend to the limit suitably.

\subsection{Siegel disks and perturbations}

Let $\alpha:=[a_0;a_1, a_2, \cdots, a_n,\cdots]$ be a Brjuno number and $q_n\geqslant 1$ be as defined in \eqref{equ:q-n}, where $n\geqslant 1$. Recall that $r_\alpha>0$ is the conformal radius of the Siegel disk $\Delta_\alpha$ of the quadratic polynomial $P_\alpha(z)=e^{2\pi\ii\alpha}z+z^2$.
Let $\phi_{\alpha}:\mathbb{D}_{r_{\alpha}}\to\Delta_{\alpha}$ be the unique conformal map with $\phi_\alpha(0)=0$ and $\phi_\alpha'(0)=1$ such that $\phi_{\alpha}\circ R_\alpha(\zeta)=P_{\alpha}\circ\phi_{\alpha}(\zeta)$ for all $\zeta\in\mathbb{D}_{r_{\alpha}}$.
We use $\lfloor x\rfloor$ to denote the integer part of a positive number $x$.

In the construction of smooth Siegel disks by Avila-Buff-Ch\'eritat, the first part of the following perturbation lemma plays a crucial role (see \cite[Main Lemma]{ABC04} or \cite{BC02}, \cite{Avi03}). The second part implies that these smooth Siegel disks are accumulated by repelling cycles (see \cite[Corollary 5, Definition 3, and \S 5]{ABC04}).

\begin{lem}[{Avila-Buff-Ch\'eritat}]\label{lem:ABC04}
For any Brjuno number $\alpha:=[a_0;a_1, a_2, \cdots$, $a_n,\cdots]$, any bounded type number $\beta:=[0;t_1,t_2,\cdots,t_n,\cdots]$ and any radius $r_0$ with $0<r_0<r_{\alpha}$, let
\begin{equation}\label{equ:alpha-n}
\alpha_n:=[a_0;a_1,a_2,\cdots,a_n, A_n,t_1,t_2,t_3,\cdots],
\end{equation}
where $A_n:=\lfloor ({r_\alpha}/{r_0})^{q_n}\rfloor$.
Then
\begin{enumerate}
\item $\alpha_n\to\alpha$ and $r_{\alpha_n}\to r_0$ as $n\to\infty$;
\item For any $\varepsilon>0$, if $n$ is large enough, then $P_{\alpha_n}$ has a repelling cycle which is $\varepsilon$-close to $\phi_{\alpha}(\T_{r_0})$ in the Hausdorff metric.
\end{enumerate}
\end{lem}

The Main Lemma in \cite{ABC04} states that Lemma \ref{lem:ABC04} holds for Brjuno numbers $(\alpha_n)_{n\geqslant 1}$. But in fact, this sequence can be chosen to be of bounded type as \eqref{equ:alpha-n} (see \cite[\S 5]{ABC04} and \cite[Remark, p.\,10]{ABC04}).
Since $\overline{P_{-\alpha}(\overline{z})}=P_\alpha(z)$, Lemma \ref{lem:ABC04} holds for $P_{-\alpha}$ with the sequence $(-\alpha_n)_{n\geqslant 1}$.

\medskip
We shall use the following fact, which implies that the conformal radius of the Siegel disk of $P_\alpha$ varies upper semicontinuously (see \cite[p.\,4]{ABC04}):
For any Brjuno number $\alpha$, if $\alpha_n\to\alpha$ and $r_{\alpha_n} \geqslant \rho$ for all $n$, then
\begin{equation}\label{equ:conf-radius}
r_\alpha \geqslant \rho\text{\quad and\quad} \phi_{\alpha_n} \rightarrow \phi_{\alpha}
\end{equation}
uniformly on compact subsets of $\D_{\rho}$ as $n\to\infty$.
Indeed, the linearizing maps $\phi_{\alpha_n}|_{\D_{\rho}}$ are univalent functions satisfying $\phi_{\alpha_n}(0)=0$ and $\phi_{\alpha_n}^{\prime}(0)=1$. Thus they form a normal family by the Koebe distortion theorem. Passing to the limit in
\begin{equation}
\phi_{\alpha_n}(e^{2\pi\ii\alpha_n} \zeta)=P_{\alpha_n}(\phi_{\alpha_n}(\zeta)),
\end{equation}
we see that any subsequence limit of $(\phi_{\alpha_n})_{n \geqslant 1}$ linearizes $P_{\alpha}$ and thus coincides with $\phi_{\alpha}$ in $\D_{\rho}$ by the uniqueness of the linearizing map.

\medskip
For a given integer $N\geqslant 1$, we define the set of \textit{high type} numbers:
\begin{equation}
 \HT_N:=\big\{\alpha=[a_0;a_1,a_2,\cdots,a_n,\cdots]\in \R\setminus\Q ~|~ a_n\geqslant N \text{ for all } n\geqslant 1\big\}.
 \end{equation}
This set of irrational numbers was introduced in \cite{IS08} so that the \textit{near-parabolic renormalization} can be acted infinitely many times for $P_\alpha$ when $\alpha$ is of sufficiently high type (i.e., $N$ is sufficiently large). The related scheme is one of the key points for constructing quadratic Julia sets with positive area (see \cite{BC12} and \cite{AL22}).
For further developments in this direction, see \cite{Che13}, \cite{CC15}, \cite{AC18}, \cite{Che19}, \cite{Che22b} and the references therein.

\begin{defi}
For a Brjuno number $\alpha$ and $0<\rho \leqslant r_{\alpha}$, we denote by $\Delta_{\alpha}(\rho)$ the invariant sub-disk of $\Delta_\alpha$ with conformal radius $\rho$ and by $L_{\alpha}(\rho)$ the set of points in $K_\alpha$ whose orbits do not intersect $\Delta_{\alpha}(\rho)$, where $K_\alpha$ is the \textit{filled-in Julia set} of $P_\alpha$.
\end{defi}

Buff and Ch\'eritat proved the existence of quadratic Siegel polynomials having a Julia set of positive area, and the following result is crucial in that proof (see \cite[Proposition 19, p.\,733]{BC12}).

\begin{lem}[{Buff-Ch\'eritat}]\label{lem:ABC12}
For sufficiently large $N\geqslant 1$, let $\alpha:=[a_0;a_1, a_2, \cdots$, $a_n,\cdots]\in\HT_N$ be a Brjuno number. For any $0<\rho<r_0<r_{\alpha}$ and $n\geqslant 1$, let
\begin{equation}\label{equ:alpha-n-new}
\alpha_n:=[a_0;a_1,a_2,\cdots,a_n, A_n,N,N,N,\cdots],
\end{equation}
where $A_n:=\lfloor ({r_\alpha}/{r_0})^{q_n}\rfloor$. Then for any $\varepsilon>0$,  if $n$ is large, then\,\footnote{By Lemma \ref{lem:ABC04}(a), $r_{\alpha_n}>\rho$ for all large $n$.}
\begin{equation}
\area(L_{\alpha_n}(\rho))\geqslant (1-\varepsilon)\area(L_\alpha(\rho)).
\end{equation}
\end{lem}

In the following, we fix the large number $N\geqslant 1$ such that Lemma \ref{lem:ABC12} holds.

\begin{defi}
For any Brjuno number $\alpha$ and any $0<\rho \leqslant r_{\alpha}$, we denote by
\begin{equation}
M_{\alpha}(\rho):=K_\alpha\setminus L_\alpha(\rho)=\big\{z\in K_\alpha: \exists\, k\geqslant 0 \text{ such that } P_\alpha^{\circ k}(z)\in \Delta_\alpha(\rho)\big\}.
\end{equation}
For any $k\geqslant 0$, we denote
\begin{equation}
M_{\alpha}^k(\rho):=\big\{z\in K_\alpha: P_\alpha^{\circ k}(z)\in \Delta_\alpha(\rho)\big\}\subset M_{\alpha}(\rho).
\end{equation}
\end{defi}

\begin{lem}\label{lem:control-loss}
Let $\alpha:=[a_0;a_1, a_2, \cdots$, $a_n,\cdots]\in\HT_N$ be a Brjuno number. For any $0<\rho<r_0<r_{\alpha}$, any $\varepsilon>0$ and any $k_0\geqslant 0$,  if $n$ is large, then
\begin{equation}
\area\big(M_{\alpha_n}(\rho)\setminus M_{\alpha_n}^{k_0}(\rho)\big)<\area\big(M_{\alpha}(\rho)\setminus M_\alpha^{k_0}(\rho)\big)+\varepsilon.
\end{equation}
where the sequence $(\alpha_n)_{n\geqslant 1}$ is defined in \eqref{equ:alpha-n-new}.
\end{lem}

\begin{proof}
Let $\varepsilon>0$ be given. By \cite[Theorem 5.1]{Dou94}, for any neighborhood $V$ of $K_\alpha$, we have $K_{\alpha_n}\subset V$ for $n$ large enough. Hence for all large $n$, we have
\begin{equation}
\area(K_{\alpha_n})< \area(K_\alpha)+\varepsilon/3.
\end{equation}
This is equivalent to
\begin{equation}\label{equ:area-1}
\area(L_{\alpha_n}(\rho))+\area(M_{\alpha_n}(\rho)) < \area(L_{\alpha}(\rho))+\area(M_{\alpha}(\rho))+\varepsilon/3.
\end{equation}
A direct calculation shows that $K_{\alpha'}\subset\overline{\D}_2$ and hence $\area(K_{\alpha'})\leqslant 4\pi$ for all $\alpha'\in\R$.
By Lemma \ref{lem:ABC12}, for all large $n$, we have
\begin{equation}
\area(L_{\alpha_n}(\rho)) > \big(1-\tfrac{\varepsilon}{12\pi}\big)\,\area(L_\alpha(\rho)).
\end{equation}
Then by \eqref{equ:area-1}, we have
\begin{equation}
\area(M_{\alpha_n}(\rho)) < \area(M_{\alpha}(\rho))+2\varepsilon/3.
\end{equation}

Let $k_0\geqslant 0$ be given. By \eqref{equ:conf-radius}, if $n$ is large, then
\begin{equation}
\big|\area(M_{\alpha_n}^{k_0}(\rho))-\area(M_{\alpha}^{k_0}(\rho))\big|< \varepsilon/3.
\end{equation}
Hence $\area\big(M_{\alpha_n}(\rho)\setminus M_{\alpha_n}^{k_0}(\rho)\big)<\area\big(M_{\alpha}(\rho)\setminus M_\alpha^{k_0}(\rho)\big)+\varepsilon$.
\end{proof}

\subsection{The surgery's continuity}

For any $\alpha\in\PZ$, $0<r<r_\alpha$ and $\theta\in[0,2\pi)$, let $A_\alpha:=A_{\alpha,r,\theta}$ be the Herman ring of the rational map $Q_{\alpha,r,\theta}$ obtained in \S\ref{subsec:rigidity}.
Note that $A_\alpha$ has modulus $\frac{1}{\pi}\log\frac{r_\alpha}{r}$ and the conformal angle between the two critical points $1\in\partial_+A_\alpha$ and $c\in\partial_-A_\alpha$ with respect to $A_\alpha$ is $\theta$.
We use
\begin{equation}\label{equ:psi-alpha-HR}
\chi_{\alpha,r,\theta}:\A_{r/r_\alpha,r_\alpha/r}\to A_\alpha
\end{equation}
to denote a conformal map such that
\begin{equation}\label{equ:linearization-HR}
\chi_{\alpha,r,\theta}(e^{2\pi\ii\alpha} \zeta)=Q_{\alpha,r,\theta}\circ\chi_{\alpha,r,\theta}(\zeta) \text{\quad for all } \zeta\in \A_{r/r_\alpha,r_\alpha/r}.
\end{equation}
Note that $\chi_{\alpha,r,\theta}$ is unique up to\footnote{In fact, one can obtain a unique $\chi_{\alpha,r,\theta}$ by requiring $\chi_{\alpha,r,\theta}(\frac{r_\alpha}{r})=1$. However, under this normalization, we cannot obtain $\chi_{\alpha_n,r,\theta}\to \chi_{\alpha,r,\theta}\,(n\to\infty)$ in Lemma \ref{lem:Herman}(b).} pre-compositing a rigid rotation of $\A_{r/r_\alpha,r_\alpha/r}$. In other words, if $\widetilde{\chi}_{\alpha,r,\theta}$ is another conformal map satisfying the above conditions, then there exists $\beta\in\R$ such that $\widetilde{\chi}_{\alpha,r,\theta}(\zeta)=\chi_{\alpha,r,\theta}(e^{\ii\beta}\zeta)$ for all $\zeta\in\A_{r/r_\alpha,r_\alpha/r}$.

\medskip
The following lemma is crucial for constructing smooth degenerate Herman rings, which is the Herman ring version of Lemma \ref{lem:ABC04}.

\begin{lem}\label{lem:Herman}
For given $\alpha\in\PZ\cap\HT_N$, $0<r<r_0<r_\alpha$ and $\theta\in[0,2\pi)$, let $\chi_{\alpha,r,\theta}$ be a conformal map in \eqref{equ:psi-alpha-HR} and $(\alpha_n)_{n\geqslant 1}$ be the sequence of bounded type numbers defined in \eqref{equ:alpha-n-new}. Then
\begin{enumerate}
\item $\alpha_n\to\alpha$, $r_{\alpha_n}\to r_0$ and $Q_{\alpha_n,r,\theta}\to Q_{\alpha,r,\theta}$ uniformly on $\EC$ as $n\to\infty$;
\item There exists a sequence $\{\chi_{\alpha_n,r,\theta}\}_{n\geqslant 1}$ such that $\chi_{\alpha_n,r,\theta}\to \chi_{\alpha,r,\theta}$ uniformly in $\A_{r/\rho,\,\rho/r}$ as large $n\to\infty$, for each $\rho\in(r,r_0)$;
\item For any $\varepsilon>0$, if $n$ is large, then $Q_{\alpha_n,r,\theta}$ has two repelling cycles which are $\varepsilon$-close to $\chi_{\alpha,r,\theta}(\T_{r_0/r})$ and $\chi_{\alpha,r,\theta}(\T_{r/r_0})$ respectively in the Hausdorff metric.
\end{enumerate}
\end{lem}

\begin{proof}
(a) By Lemma \ref{lem:ABC04}(a), $\alpha_n\to\alpha$ and $r_{\alpha_n}\to r_0$ as $n\to\infty$. We take $r<r'<\rho<r_0$. Shifting the subscript of $(\alpha_n)_{n\geqslant 1}$ if necessary, without loss of generality we assume that
\begin{equation}
r_{\alpha_n}>\rho>r'>r \text{\quad for all } n\geqslant 1.
\end{equation}
We continue using the same notations as in \S\ref{subsec:S-to-H}.
The conformal maps
\begin{equation}\label{equ:psi-n-alpha-pm}
\begin{split}
\psi_{\alpha_n,+}(\zeta):=\phi_{\alpha_n}\big(r \zeta\big): &~ \A_{1,r_{\alpha_n}/r} \to A(\gamma_{\alpha_n,r},\partial\Delta_{\alpha_n}) \text{\quad and\quad} \\
\psi_{\alpha_n,-}(\zeta):= \eta_r\circ\phi_{-\alpha_n}\Big(\tfrac{r e^{\ii\theta}}{\zeta}\Big): &~ \A_{r/r_{\alpha_n},r/r'} \to A(\eta_r(\partial\Delta_{-\alpha_n}),\eta_r(\gamma_{-\alpha_n,r'}))
\end{split}
\end{equation}
satisfy
\begin{equation}\label{equ:psi-n-pm}
\begin{split}
\psi_{\alpha_n,+}^{-1}\circ P_{\alpha_n}\circ\psi_{\alpha_n,+}(\zeta)=&~R_{\alpha_n}(\zeta) \text{\quad for all } \zeta\in\A_{1,r_{\alpha_n}/r} \text{\quad and} \\
\psi_{\alpha_n,-}^{-1}\circ (\eta_r\circ P_{-\alpha_n}\circ\eta_r^{-1})\circ\psi_{\alpha_n,-}(\zeta)=&~R_{\alpha_n}(\zeta) \text{\quad for all } \zeta\in\A_{r/r_{\alpha_n},r/r'}.
\end{split}
\end{equation}
By \eqref{equ:conf-radius}, $\phi_{\alpha_n}\to\phi_\alpha$ and $\phi_{-\alpha_n}\to\phi_{-\alpha}$ uniformly on $\overline{\D}_{\rho}$ as $n\to\infty$. We have
\begin{equation}\label{equ:psi-n-convg}
\psi_{\alpha_n,+}\to\psi_{\alpha,+} \text{\quad and\quad} \psi_{\alpha_n,-}\to\psi_{\alpha,-}
\end{equation}
uniformly on $\overline{\A}_{1,\,\rho/r}$ and $\overline{\A}_{r/\rho,\,r/r'}$ respectively as $n\to\infty$, where $\psi_{\alpha,+}$ and $\psi_{\alpha,-}$ are defined in \eqref{equ:psi-pm-std}.
There exists an interpolation
\begin{equation}\label{equ:psi-0-n}
\psi_{\alpha_n,0}:\overline{\A}_{r/r',1} \to \overline{A(\eta_r(\gamma_{-\alpha_n,r'}),\gamma_{\alpha_n,r})}
\end{equation}
such that $\Psi_{\alpha_n}: \A_{r/r_{\alpha_n},r_{\alpha_n}/r} \to A(\eta_r(\partial\Delta_{-\alpha_n}),\partial\Delta_{\alpha_n})$ is a $C^\infty$-diffeomorphism, where
\begin{equation}
\Psi_{\alpha_n}(\zeta):=
\left\{
\begin{array}{ll}
\psi_{\alpha_n,+}(\zeta)  &~~~~~~~\text{if}~\zeta\in \A_{1,r_{\alpha_n}/r}, \\
\psi_{\alpha_n,-}(\zeta) &~~~~~~\text{if}~\zeta\in  \A_{r/r_{\alpha_n},r/r'}, \\
\psi_{\alpha_n,0}(\zeta) &~~~~~~\text{if}~\zeta\in \overline{\A}_{r/r',1}.
\end{array}
\right.
\end{equation}
By \eqref{equ:psi-n-convg}, $\psi_{\alpha_n,0}$ can be chosen such that
\begin{equation}\label{equ:Psi-n-conv}
\psi_{\alpha_n,0}\to \psi_{\alpha,0} \text{\quad and\quad}\Psi_{\alpha_n}\to\Psi_{\alpha}
\end{equation}
uniformly on $\overline{\A}_{r/r',1}$ and $\overline{\A}_{r/\rho,\,\rho/r}$ respectively as $n\to\infty$, where $\psi_{\alpha,0}$ and $\Psi_\alpha$ are defined in \eqref{equ:psi-0} and \eqref{equ:Psi} respectively.

\medskip
Define
\begin{equation}\label{equ:F-n-quasiregular}
F_{\alpha_n}(z):=
\left\{
\begin{array}{ll}
P_{\alpha_n}(z)  &~~~~~~~\text{if}~z\in \gamma_{\alpha_n,r}^\Ext, \\
\eta_r\circ P_{-\alpha_n}\circ\eta_r^{-1}(z) &~~~~~~\text{if}~z\in \eta_r(\gamma_{-\alpha_n,r'}^\Ext), \\
\psi_{\alpha_n,0}\circ R_{\alpha_n}\circ\psi_{\alpha_n,0}^{-1}(z) &~~~~~~\text{if}~z\in \overline{A(\eta_r(\gamma_{-\alpha_n,r'}),\gamma_{\alpha_n,r})}.
\end{array}
\right.
\end{equation}
Let $F_\alpha$ be as defined in \eqref{equ:F-quasiregular}. By \eqref{equ:Psi-n-conv}, we have
\begin{equation}\label{equ:F-n}
F_{\alpha_n}\to F_\alpha
\end{equation}
uniformly on $\EC$ as $n\to\infty$.
Note that $F_{\alpha_n}:\EC\to\EC$ is a $C^\infty$-branched covering and hence a quasi-regular map which is analytic outside $\overline{A(\eta_r(\gamma_{-\alpha_n,r'}),\gamma_{\alpha_n,r})}$.
We pull back the standard complex structure $\sigma_0$ in $\A_{r/r_{\alpha_n},r_{\alpha_n}/r}$ to define an almost complex structure $\sigma_{\alpha_n}:=(\Psi_{\alpha_n}^{-1})^*(\sigma_0)$ in $A_{\alpha_n}':=A(\eta_r(\partial\Delta_{-\alpha_n}),\partial\Delta_{\alpha_n})$. Similar to Lemma \ref{lem:surgery-S-to-H}, $\sigma_{\alpha_n}$ is invariant under $F_{\alpha_n}: A_{\alpha_n}'\to A_{\alpha_n}'$.

Next, we extend $\sigma_{\alpha_n}$ to $F_{\alpha_n}^{-k}(A_{\alpha_n}')$ by defining $\sigma_{\alpha_n}:=(F_{\alpha_n}^{\circ k})^*(\sigma_{\alpha_n})$ in $F_{\alpha_n}^{-k}(A_{\alpha_n}')\setminus F_{\alpha_n}^{-(k-1)}(A_{\alpha_n}')$ for all $k\geqslant 1$. In the rest place we define $\sigma_{\alpha_n}:=\sigma_0$. Note that $\sigma_{\alpha_n}$ has uniformly bounded dilatation. By the measurable Riemann mapping theorem, there exists a unique quasiconformal mapping $\Phi_{\alpha_n}:\EC\to\EC$ integrating the almost complex structure $\sigma_{\alpha_n}$ such that
\begin{itemize}
\item $\Phi_{\alpha_n}^*(\sigma_0)=\sigma_{\alpha_n}$ and $\Phi_{\alpha_n}:(\EC,\sigma_{\alpha_n})\rightarrow (\EC,\sigma_0)$ is an analytic isomorphism;
\item  $\Phi_{\alpha_n}(0)=0$, $\Phi_{\alpha_n}(\infty)=\infty$ and $\Phi_{\alpha_n}(\omega_{\alpha_n})=1$, where $\omega_{\alpha_n}=-e^{2\pi\ii\alpha_n}/2$ is the critical point of $P_{\alpha_n}$.
\end{itemize}
Hence
\begin{equation}\label{equ:Q-n-r-theta}
Q_{\alpha_n,r,\theta}:=\Phi_{\alpha_n}\circ F_{\alpha_n}\circ \Phi_{\alpha_n}^{-1}:(\EC,\sigma_0)\to(\EC,\sigma_0)
\end{equation}
is a cubic rational map having two super-attracting fixed points $0$ and $\infty$, two different critical points $1$ and $c_n:=\Phi_{\alpha_n}(\eta_r(\omega_{-\alpha_n}))$ in $\C$ and a fixed Herman ring $A_{\alpha_n}:=\Phi_{\alpha_n}(A_{\alpha_n}')$ with modulus $\frac{1}{\pi}\log\frac{r_{\alpha_n}}{r}$.

\medskip
Let $\sigma_\alpha$ be as introduced in the proof of Lemma \ref{lem:surgery-S-to-H}. We claim that
\begin{equation}\label{equ:sigma-n-cont}
\|\sigma_{\alpha_n}-\sigma_\alpha\|_\infty\to 0 \text{\quad as }n \to\infty.
\end{equation}
Let $\Phi_\alpha:\EC\to\EC$ be the unique quasiconformal mapping satisfying $\Phi_\alpha^*(\sigma_0)=\sigma_\alpha$, $\Phi_\alpha(0)=0$, $\Phi_\alpha(\infty)=\infty$ and $\Phi_\alpha(\omega_\alpha)=1$. Note that $\omega_{\alpha_n}\to\omega_\alpha$ as $n\to\infty$. By \eqref{equ:sigma-n-cont}, $\Phi_{\alpha_n}\to \Phi_\alpha$ uniformly on $\EC$ as $n\to\infty$.
By \eqref{equ:F-n} and \eqref{equ:Q-n-r-theta}, we have $Q_{\alpha_n,r,\theta}\to Q_{\alpha,r,\theta}$ as $n\to\infty$.

\medskip
To prove \eqref{equ:sigma-n-cont}, it suffices to prove that for any $\varepsilon>0$ and $\delta>0$, if $n$ is large, then $\area(X_n^\delta)<\varepsilon$, where
\begin{equation}
X_n^\delta:=\big\{z\in\C\,:\,\big|\sigma_{\alpha_n}(z)-\sigma_\alpha(z)\big|>\delta\big\}.
\end{equation}
Note that $F_{\alpha_n}:\EC\to\EC$ and $F_{\alpha}:\EC\to\EC$ are analytic in $\EC\setminus W_{\alpha_n}$ and $\EC\setminus W_\alpha$ respectively, where
\begin{equation}
W_{\alpha_n}:= \psi_{\alpha_n,0}(\overline{\A}_{r/r',1}) \text{\quad and\quad} W_\alpha:=\psi_{\alpha,0}(\overline{\A}_{r/r',1}).
\end{equation}
By the definition of $\sigma_{\alpha_n}$ and $\sigma_\alpha$, for any $n\geqslant 1$, we have $\sigma_{\alpha_n}(z)=\sigma_\alpha(z)=0$ for all $z\in\EC\setminus (\MW_{n}\cup \MW)$ and
\begin{equation}\label{equ:X-n-0}
X_n^\delta\subset \MW_{n}\cup \MW,
\end{equation}
where
\begin{equation}
\MW_{n}:=\bigcup_{j\geqslant 0} F_{\alpha_n}^{-j}(W_{\alpha_n}) \text{\quad and\quad} \MW:=\bigcup_{j\geqslant 0} F_{\alpha}^{-j}(W_{\alpha}).
\end{equation}

For $n\geqslant 1$ and $k\geqslant 0$, we denote
\begin{equation}
\MW_{n}^k:=\bigcup_{0\leqslant j\leqslant k} F_{\alpha_n}^{-j}(W_{\alpha_n}) \text{\quad and\quad} \MW^k:=\bigcup_{0\leqslant j\leqslant k} F_{\alpha}^{-j}(W_{\alpha}).
\end{equation}
Note that $\area(\MW)<+\infty$, $W_{\alpha_n}\subset \Delta_{\alpha_n}(r)$ and $W_\alpha\subset \Delta_\alpha(r)$.
By the surgery construction and Lemma \ref{lem:control-loss}, there exists a large $k_0\geqslant 1$ such that for all large $n$,
\begin{equation}\label{equ:inequ-1}
\area(\MW\setminus \MW^{k_0})<\varepsilon/5  \text{\quad and\quad} \area(\MW_n\setminus \MW_n^{k_0})<\varepsilon/5.
\end{equation}

By \eqref{equ:Psi-n-conv}, $\psi_{\alpha_n,0}\to \psi_{\alpha,0}$ uniformly on $\overline{\A}_{r/r',1}$ as $n\to\infty$.
Hence there exists an open topological disk $V$, such that for large $n$, we have $\overline{V}\subset W_{\alpha_n}\cap W_\alpha$ and
\begin{equation}\label{equ:inequ-2}
\area(\MW_{n}^{k_0}\setminus \MV^{k_0})<\varepsilon/5 \text{\quad and\quad} \area(\MW^{k_0}\setminus \MV^{k_0})<\varepsilon/5,
\end{equation}
where $\MV^{k_0}:=\bigcup_{0\leqslant j\leqslant {k_0}}F_\alpha^{-j}(V)$.
By \eqref{equ:F-n}, if $n$ is large, we have
\begin{equation}\label{equ:inequ-3}
\area(X_n^\delta\cap \MV^{k_0})<\varepsilon/5.
\end{equation}
Combining \eqref{equ:X-n-0}, \eqref{equ:inequ-1}, \eqref{equ:inequ-2} and \eqref{equ:inequ-3}, for large $n$ we have $\area(X_n^\delta)<\varepsilon$. By the arbitrariness of $\varepsilon$ and $\delta$, \eqref{equ:sigma-n-cont} holds and we have $Q_{\alpha_n,r,\theta}\to Q_{\alpha,r,\theta}$ as $n\to\infty$.

\medskip
(b) Since the numbers $r\in(0,r_0)$ and $\theta\in[0,2\pi)$ are fixed as we take limit, for simplicity we use $\alpha_n$ and $\alpha$ to replace the subscripts $(\alpha_n,r,\theta)$ and $(\alpha,r,\theta)$ respectively.
The proof of this part is similar to \eqref{equ:conf-radius}.
Indeed, by Part (a), there exists a large $R>0$ such that for all $n\geqslant 1$,
\begin{equation}\label{equ:nbd-infty}
\EC\setminus\D_R\subset B_{\alpha_n}^\infty \text{\quad and\quad}\EC\setminus\D_R\subset B_\alpha^\infty,
\end{equation}
where $B_{\alpha_n}^\infty$ and $B_\alpha^\infty$ are the immediate super-attracting basins of $\infty$ with respect to $Q_{\alpha_n}$ and $Q_\alpha$ respectively.
Then the Herman ring $A_{\alpha_n}$ of $Q_{\alpha_n}$ is disjoint from $\EC\setminus\D_R$ for all $n\geqslant 1$. The restrictions $\chi_{\alpha_n}|_{\A_{r/\rho,\,\rho/r}}$ of the linearizing maps $\chi_{\alpha_n}:\A_{r/r_{\alpha_n},\,r_{\alpha_n}/r}\to A_{\alpha_n}$  are well-defined for all $n\geqslant 1$, and thus form a normal family. Passing to the limit in
\begin{equation}
\chi_{\alpha_n}(e^{2\pi\ii\alpha_n} \zeta)=Q_{\alpha_n}(\chi_{\alpha_n}(\zeta)),
\end{equation}
we see that any subsequence limit $\widetilde{\chi}_\alpha$ of $(\chi_{\alpha_n})_{n \geqslant 1}$ linearizes $Q_{\alpha}$.
By the uniqueness of the linearizing map, there exists $\beta\in\R$ such that $\widetilde{\chi}_\alpha(\zeta)=\chi_\alpha(e^{\ii\beta}\zeta)$ for all $\zeta\in\A_{r/\rho,\,\rho/r}$.
In particular, $\chi_{\alpha_n}(\T)\to \chi_{\alpha}(\T)$ as $n\to\infty$ in the Hausdorff metric.
Therefore, by pre-compositing a rigid rotation if necessary, we may assume that $\{\chi_{\alpha_n}\}_{n\geqslant 1}$ is chosen such that $\chi_{\alpha_n}(1)\in\chi_{\alpha_n}(\T)$ satisfies $\chi_{\alpha_n}(1)\to \chi_\alpha(1)$ as $n\to\infty$.
By the arbitrariness of $\rho\in(r,r_0)$, this implies that $\chi_{\alpha_n}\to \chi_{\alpha}$ uniformly in $\A_{r/\rho,\,\rho/r}$ as $n\to\infty$.

\medskip
(c) A direct calculation shows that the filled-in Julia sets of $P_{\alpha_n}$ and $P_\alpha$ are contained in the closed disk $\overline{\D}_2$.
Note that $(\Phi_{\alpha_n})_{n\geqslant 1}$ are quasiconformal mappings having uniformly bounded dilatation and satisfying $\Phi_{\alpha_n}(0)=0$, $\Phi_{\alpha_n}(\infty)=\infty$ and $\Phi_{\alpha_n}(\omega_{\alpha_n})=1$, where $\omega_{\alpha_n}\to\omega_\alpha$ as $n\to\infty$.
By Mori's theorem (see e.g., \cite[p.\,66]{LV73}), there exist constants $K_0$, $K>1$ such that for all $z\in\overline{\D}_2$ and $n\geqslant 1$, we have
\begin{equation}\label{equ:distortion-qc}
|\Phi_{\alpha_n}(z_1)-\Phi_{\alpha_n}(z_2)|<K_0 |z_1-z_2|^{\frac{1}{K}}.
\end{equation}
By \eqref{equ:Psi-Phi}, the Herman ring $A_\alpha$ of $Q_\alpha$ has modulus $\frac{1}{\pi}\log\frac{r_{\alpha}}{r}$ and $\Phi_{\alpha}\circ\Psi_{\alpha}:\A_{r/r_{\alpha},r_{\alpha}/r} \to A_{\alpha}$ is a conformal isomorphism. By the uniqueness of the linearizing map, there exists a number $\beta\in\R$ such that
\begin{equation}\label{equ:chi-alpha-coor}
\chi_\alpha(\zeta)=\Phi_\alpha\circ\Psi_\alpha(e^{\ii \beta}\zeta) \text{\quad for all }\zeta\in \A_{r/r_{\alpha},r_{\alpha}/r}.
\end{equation}

Let $\varepsilon>0$ be given. Without loss of generality we only prove that for large $n$, $Q_{\alpha_n}$ has a repelling cycle which is $\varepsilon$-close to $\chi_\alpha(\T_{r_0/r})$ in the Hausdorff metric since the proof for $\chi_\alpha(\T_{r/r_0})$ is completely similar. By Lemma \ref{lem:ABC04}(b), if $n$ is large, $P_{\alpha_n}$ has a repelling cycle $C_n\subset\overline{\D}_2$ which is $\big(\frac{\varepsilon}{2K_0}\big)^K$-close to $\phi_{\alpha}(\T_{r_0})$. By \eqref{equ:distortion-qc}, $\Phi_{\alpha_n}\circ P_{\alpha_n}\circ \Phi_{\alpha_n}^{-1}$ has a repelling cycle $\Phi_{\alpha_n}(C_n)$ which is $\tfrac{\varepsilon}{2}$-close to $\Phi_{\alpha_n}\circ\phi_{\alpha}(\T_{r_0})$.
Since $\Phi_{\alpha_n}\to \Phi_\alpha$ uniformly on $\EC$ as $n\to\infty$, by \eqref{equ:psi-alpha-pm}, \eqref{equ:Psi} and \eqref{equ:chi-alpha-coor}, for all large $n$, $\Phi_{\alpha_n}\circ\phi_{\alpha}(\T_{r_0})$ is $\tfrac{\varepsilon}{2}$-close to
\begin{equation}
\Phi_{\alpha}\circ\phi_{\alpha}(\T_{r_0})=\Phi_{\alpha}\circ\Psi_{\alpha}(\T_{r_0/r})=\chi_\alpha(\T_{r_0/r}).
\end{equation}
According to \eqref{equ:F-n-quasiregular} and \eqref{equ:Q-n-r-theta}, for all large $n$, $Q_{\alpha_n}$ has a repelling cycle which is $\varepsilon$-close to $\chi_\alpha(\T_{r_0/r})$ in the Hausdorff metric.
\end{proof}

Let $(Q_{\alpha_n,r,\theta})_{n\geqslant 1}$ be the sequence of cubic rational maps obtained in Lemma \ref{lem:Herman}.
By Lemma \ref{lem:rigidity}, there exist unique pairs $(a, u)$ and $(a_n,u_n)$ for each $n\geqslant 1$ such that
\begin{equation}\label{equ:Q-a-n-u-n}
Q_{a_n,u_n}=Q_{\alpha_n,r,\theta}, \quad Q_{a,u}=Q_{\alpha,r,\theta}\text{\quad and\quad} (a_n,u_n)\to(a,u) \text{ as }n\to\infty.
\end{equation}

Note that it is hard to find the explicit parameter pair $(a,u)$ for given $(\alpha, r,\theta)$ in general. The purpose of Figure \ref{Fig_Herman-sequence} is to give a \textit{rough} illustration of Lemma \ref{lem:Herman}. We take $\alpha=(\sqrt{5}-1)/2=[0;1,1,1,\cdots]$, $N=1$ and consider the sequence $(\alpha_n)_{n\geqslant 1}$ in \eqref{equ:alpha-n-new} such that $\alpha_n\to\alpha$ and $r_{\alpha_n}\to r_0=\frac{10}{13}r_\alpha$  as $n\to\infty$.

\begin{figure}[tpb]
  \setlength{\unitlength}{1mm}
  \fbox{\includegraphics[width=0.42\textwidth]{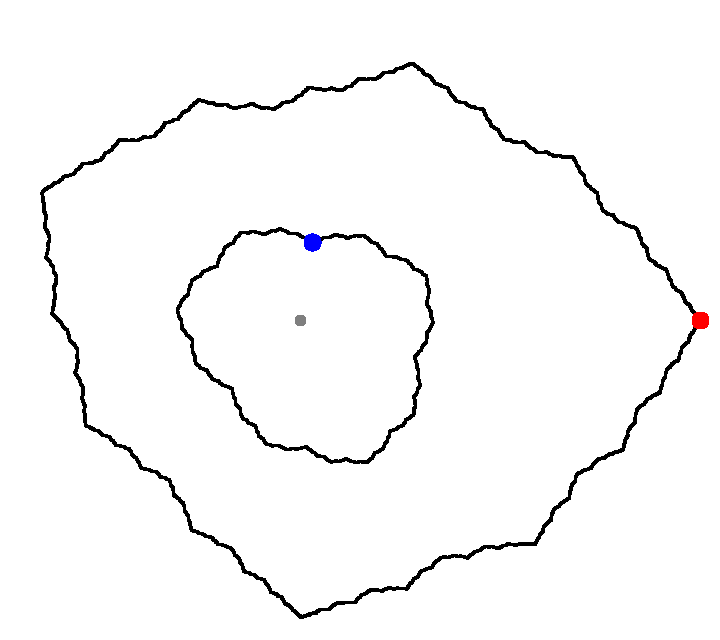}} \quad
  \fbox{\includegraphics[width=0.42\textwidth]{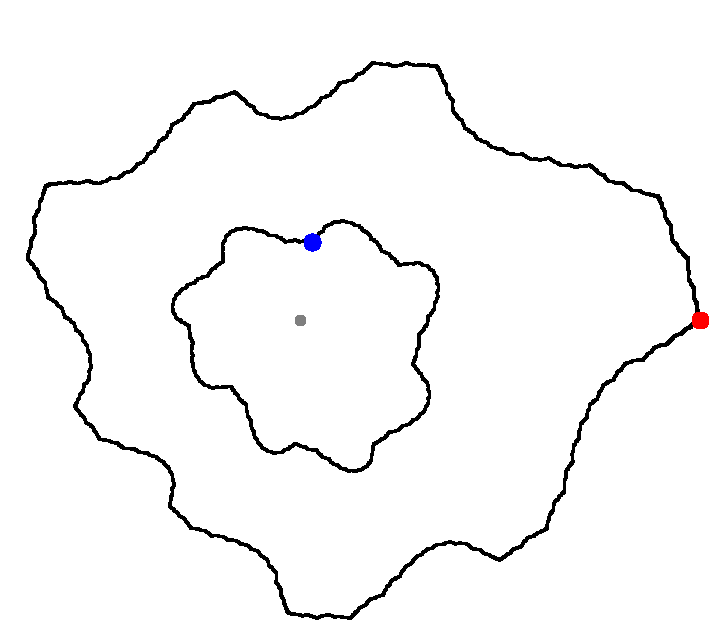}}
  \put(-4,22){$1$}
  \put(-36.8,22){$0$}
  \put(-34.5,30){$c$}
  \put(-70,28.5){$1$}
  \put(-104,22){$0$}
  \put(-102,30){$c$} \vskip0.3cm
  \fbox{\includegraphics[width=0.42\textwidth]{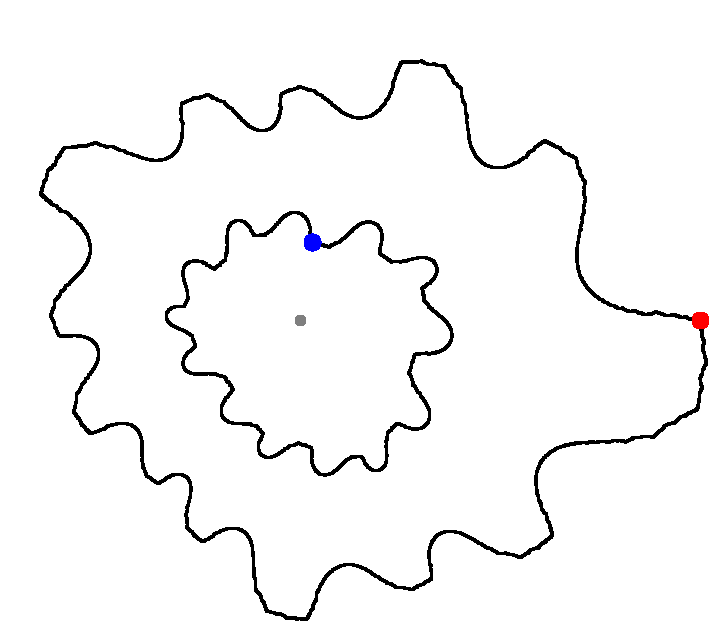}} \quad
  \fbox{\includegraphics[width=0.42\textwidth]{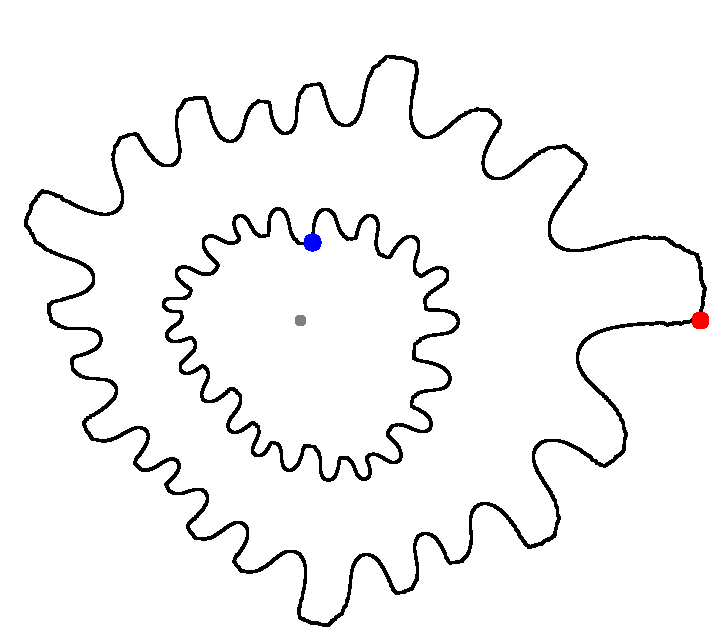}}
  \put(-4,22){$1$}
  \put(-36.8,22){$0$}
  \put(-34.5,30){$c$}
  \put(-71,28.5){$1$}
  \put(-104,22){$0$}
  \put(-102,29.8){$c$}
  \caption{The boundaries of Herman rings of $Q_{a,u}$ with $\alpha=(\sqrt{5}-1)/2$ and the corresponding sequence $Q_{a_n,u_n}$ for $n=5,6,7$ (from top left to bottom right). As a \textit{rough} illustration of Lemma \ref{lem:Herman}, we fix the choice of the parameter $a=a_n=2+\frac{\ii}{10}$. One may observe that the shape and the conformal moduli of these annuli tend to some limit (compare \cite[Fig. 4, p.\,25]{ABC04} for the case of Siegel disks).}
  \label{Fig_Herman-sequence}
\end{figure}
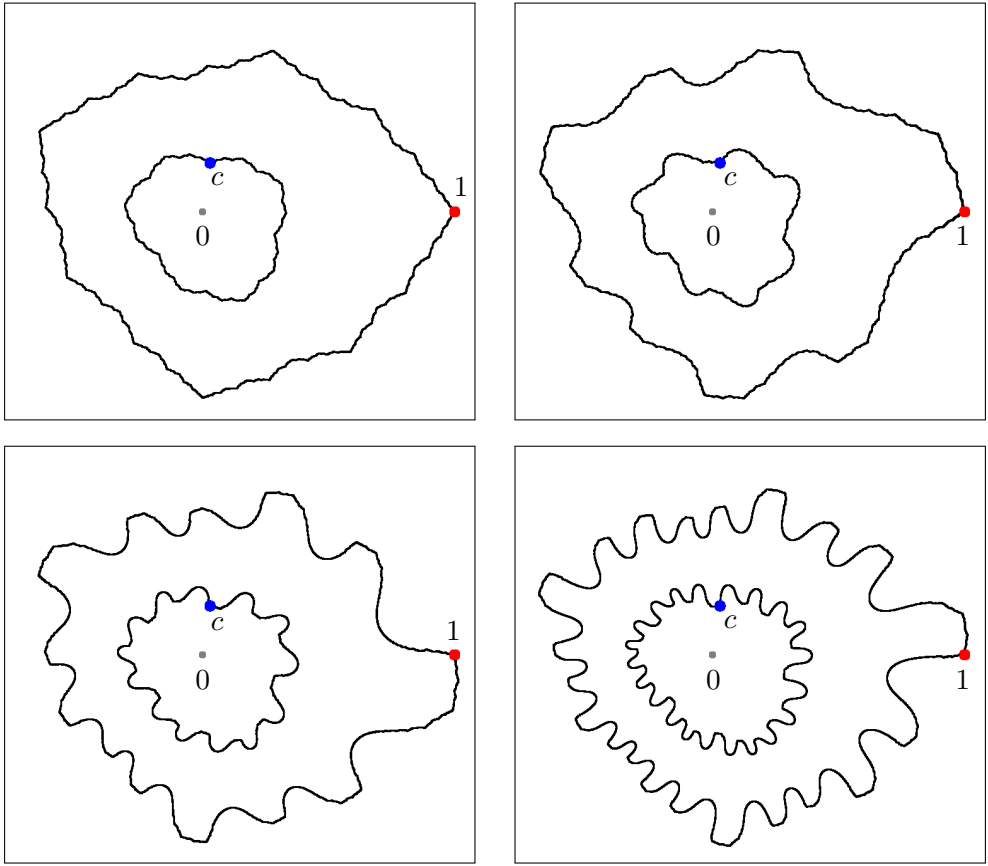

To obtain Figure \ref{Fig_Herman-sequence}, we first draw the $u$-parameter space of $Q_{a,u}$ globally by fixing some $a$ which\footnote{By a similar idea of \cite[Proof of Theorem B, p.\,11]{Yan22}, one can show that for any given Brjuno number $\alpha$, if $|a-2|$ is sufficiently small, then there exists a parameter $u$ such that $Q_{a,u}$ has a fixed Herman ring of rotation number $\alpha$. This strategy has been used in \cite[p.\,993]{Yan22} to generate pictures of periodic Herman rings.} is close to $2$ (Here we choose $a=2+\frac{\ii}{10}$). One can obtain the approximate parameter $u$ (resp. $u_n$) corresponding to the given rotation number $\alpha$ (resp. $\alpha_n$) by zooming the bifurcation locus of $Q_{a,u}$ (resp. $Q_{a,u_n}$).

\section{Proofs of Theorems \ref{thm:main-1} and \ref{thm:main-2}} \label{sec:proof-ThmA}

\subsection{Proof of Theorem \ref{thm:main-1}}

Any complete metrizable locally convex vector space is called a \textit{Fr\'{e}chet space}. All Banach spaces are Fr\'{e}chet. In particular, the following space is a typical Fr\'{e}chet space (see \cite[\S 1.46, p.\,34]{Rud91}):
\begin{equation}\label{equ:C-infty}
C^\infty(\R/\Z,\C):=\big\{f:\R/\Z\to\C \text{ is a } C^\infty \text{-function}\big\}.
\end{equation}
Specifically, the metric in $C^\infty(\R/\Z,\C)$ is defined as
\begin{equation}
d(f,g):=\sum_{k=0}^\infty 2^{-k}\frac{\|f-g\|_k}{1+\|f-g\|_k}, \quad\text{for } f,g\in C^\infty(\R/\Z,\C),
\end{equation}
where the seminorm $\|f\|_k$ with $k\geqslant 0$ is defined by
\begin{equation}
\|f\|_k:=\sup\{|f^{(k)}(x)|:x\in\R/\Z\}.
\end{equation}
In this Fr\'{e}chet space $C^\infty(\R/\Z,\C)$, a sequence $(f_n)$ converges to $f$ if and only if for each $k\geqslant 0$, the sequence $(f_n^{(k)})$ converges to $f^{(k)}$ uniformly.
For $0<\varepsilon<1$, two maps $f,g\in C^\infty(\R/\Z,\C)$ are said to be \textit{$\varepsilon$-close} if $d(f,g)<\varepsilon$, and in this case $f(\R/\Z)$ and $g(\R/\Z)$ are $\frac{\varepsilon}{1-\varepsilon}$-close in the Hausdorff metric.

\medskip
Recall that $\HT_N$ is the set of high type irrational numbers and $N\geqslant 1$ is a fixed large number such that Lemma \ref{lem:ABC12} holds.
For $\alpha\in\PZ$, $0<r<r_\alpha$ and $\theta\in[0,2\pi)$, by Lemma \ref{lem:rigidity}, there exists a unique pair $(a, u)$ with $a\in\C\setminus\big\{0,1,\frac{3}{2},2,3\big\}$ and $u\in\C\setminus\{0\}$ such that $Q_{a,u}=Q_{\alpha,r,\theta}$.

\begin{thm}\label{thm:A-detail}
Suppose $\alpha\in\HT_N$ is of bounded type, $0<r<r_\alpha$, $\theta\in(0,2\pi)$ and $0<\varepsilon<1$. Let $g:\R/\Z\to\C$ be the function $t\mapsto \chi_{\alpha,r,\theta}(e^{2\pi\ii t})$, where $\chi_{\alpha,r,\theta}$ is a linearizing map of the Herman ring of $Q_{\alpha,r,\theta}=Q_{a,u}$ in \eqref{equ:psi-alpha-HR}. Then, there exist an irrational number $\alpha'\in\HT_N$, a cubic rational map $Q_{a',u'}$ and a $C^\infty$-embedding $h:\R/\Z\to \C$ with the following properties:
\begin{itemize}
\item $|\alpha-\alpha'|<\varepsilon$, $|a-a'|<\varepsilon$ and $|u-u'|<\varepsilon$;
\item $h(\R/\Z)$ is a smooth degenerate Herman ring of $Q_{a',u'}$ with rotation number $\alpha'$ which is accumulated by repelling cycles from both sides; and
\item $g$ and $h$ are $\varepsilon$-close in the Fr\'{e}chet space $C^\infty(\R/\Z,\C)$.
\end{itemize}
\end{thm}

\begin{proof}
We define sequences of bounded type numbers $(\alpha(n))_{n\geqslant 0}$ in $\HT_N$, positive numbers $(\varepsilon_{n})_{n\geqslant 0}$ and $(r_{n})_{n\geqslant 0}$, cubic rational maps $(Q_{\alpha(n),r,\theta})_{n\geqslant 0}=(Q_{a(n),u(n)})_{n\geqslant 0}$ with the form \eqref{equ:Q-alpha}, and repelling cycles $(C_n^\pm)_{n\geqslant 0}$ inductively as follows.

Take $\alpha(0):=\alpha$, $\varepsilon_{0}:=\frac{1}{10} \varepsilon$, $r_{0}:=r_{\alpha}$, $Q_{\alpha(0)}:=Q_{\alpha(0),r,\theta}=Q_{a(0),u(0)}$, and let $C_0^\pm$ be the two repelling fixed points of $Q_{\alpha(0)}$ which lie on the boundaries of immediate super-attracting basins of $\infty$ and $0$ respectively.
Taking $\varepsilon>0$ smaller if necessary, we assume that $|a_*-a(0)|+|u_*-u(0)|>\varepsilon$ for any singular point\footnote{Based on the formula of $Q_{a,u}$ in Lemma \ref{lem:rigidity}, there are only finitely many choices for $(a_*,u_*)$.} $(a_*,u_*)\in\C^2$ such that $Q_{a_*,u_*}$ is not a cubic rational map.

Assume that $\alpha(n)$, $\varepsilon_{n}$, $r_n$, $Q_{\alpha(n)}:=Q_{\alpha(n),r,\theta}=Q_{a(n),u(n)}$ and $C_n^\pm$ are defined, where $\min\{r_{\alpha(n)},r_n\}>r$.
Let $\varepsilon_{n+1}<\frac{1}{10} \varepsilon_{n}$ be such that
\begin{itemize}
\item $r_{\beta}<r_{\alpha(n)}+\varepsilon_{n}$ whenever $|\beta-\alpha(n)|<\varepsilon_{n+1}$ (this is possible by the upper semicontinuity, see \eqref{equ:conf-radius}); and
\item $Q_{b,v}$ has two repelling cycles which are $\varepsilon_n$-close to $C_n^\pm$ respectively whenever $|b-a(n)|+|v-u(n)|<\varepsilon_{n+1}$ (since repelling cycles move holomorphically).
\end{itemize}
Next, taking $r_{n+1}\in\big(r,\,\min\{r_{\alpha(n)},r_n\}\big)$ sufficiently close to $r$ so that
\begin{itemize}
\item $r_{n+1}-r<\varepsilon_{n+1}$; and
\item $\chi_{\alpha(n)}(\T_{r_{n+1}/r})$ and $\chi_{\alpha(n)}(\T_{r/r_{n+1}})$ are $\varepsilon_{n+1}$-close to $\chi_{\alpha(n)}(\T)$ in the Hausdorff metric, where $\chi_{\alpha(n)}:=\chi_{\alpha(n),r,\theta}$.
\end{itemize}
Finally, by Lemma \ref{lem:Herman} and \eqref{equ:Q-a-n-u-n}, there exists a bounded type $\alpha(n+1)\in\HT_N$ such that\footnote{The number $r_{n+1}$ here corresponds to $r_0$ in Lemma \ref{lem:Herman}.}
\begin{itemize}
\item $|\alpha(n+1)-\alpha(n)|<\frac{1}{10} \varepsilon_{n+1}$;
\item $|a(n+1)-a(n)|+|u(n+1)-u(n)|<\frac{1}{10}\varepsilon_{n+1}$, where $a(n+1)$ and $u(n+1)$ are parameters such that $Q_{a(n+1),u(n+1)}=Q_{\alpha(n+1),r,\theta}$;
\item $r<r_{\alpha(n+1)}<r_{n+1}+\varepsilon_{n+1}$;
\item $Q_{\alpha(n+1)}:=Q_{\alpha(n+1),r,\theta}$ has two repelling cycles $C_{n+1}^\pm$ which are $\varepsilon_{n+1}$-close to $\chi_{\alpha(n)}(\T_{r_{n+1}/r})$ and $\chi_{\alpha(n)}(\T_{r/r_{n+1}})$ respectively, and hence are $2\varepsilon_{n+1}$-close to $\chi_{\alpha(n)}(\T)$ in the Hausdorff metric; and
\item The real-analytic functions $g_{n+1}: t \mapsto \chi_{\alpha(n+1)}(e^{2 \pi\ii t})$ and $g_{n}: t \mapsto \chi_{\alpha(n)}(e^{2 \pi\ii t})$ are $\varepsilon_{n+1}$-close in the Fr\'{e}chet space $C^\infty(\R/\Z,\C)$, where $\chi_{\alpha(n+1)}:=\chi_{\alpha(n+1),r,\theta}$.
\end{itemize}

Let
\begin{equation}\label{equ:limit-alpha-a-u}
\alpha':=\lim _{n \to \infty} \alpha(n), \quad a':=\lim _{n \to \infty} a(n) \text{\quad and\quad} u':=\lim _{n \to \infty} u(n).
\end{equation}
By construction, for $n \geqslant 0$ we have
\begin{equation}\label{equ:alpha-a-u-est}
|\alpha'-\alpha(n)|<\varepsilon_{n+1} \text{\quad and\quad} |a'-a(n)|+|u'-u(n)|<\varepsilon_{n+1}.
\end{equation}
The functions $g_{n}$ converge to a $C^{\infty}$-function $h: \R/\Z \to \C$, which is $\varepsilon$-close to $g=g_{0}$ in the Fr\'{e}chet space $C^\infty(\R/\Z,\C)$. By taking $\varepsilon$ smaller if necessary, this implies that $h$ is an embedding. Passing to the limit in
\begin{equation}
\chi_{\alpha(n)}(e^{2\pi\ii\alpha(n)} \zeta)=Q_{a(n),u(n)}\circ\chi_{\alpha(n)}(\zeta)
\end{equation}
for $\zeta\in\T$, we obtain
\begin{equation}\label{equ:smooth-conj}
\chi_{\alpha'}(e^{2\pi\ii\alpha'} \zeta)=Q_{a',u'}\circ\chi_{\alpha'}(\zeta), \text{\quad where }\chi_{\alpha'}:=\lim _{n \to \infty} \chi_{\alpha(n)}|_{\T}.
\end{equation}
Hence $\chi_{\alpha'}(e^{2\pi\ii t})=h(t)$ for $t\in\R/\Z$, and $Q_{a',u'}: h(\R/\Z)\to h(\R/\Z)$ is smoothly conjugate to the irrational rotation $\zeta\mapsto e^{2\pi\ii\alpha'} \zeta$.

By \eqref{equ:alpha-a-u-est}, $Q_{a',u'}$ has two repelling cycles $\widetilde{C}_n^\pm$ which are $\varepsilon_n$-close to $C_n^\pm$ respectively. Since $C_n^\pm$ are $2\varepsilon_n$-close to $g_{n-1}(\R/\Z)$ and $h(\R/\Z)$ is $2\varepsilon_n$-close to $g_{n-1}(\R/\Z)$ in the Hausdorff metric, it follows that $\widetilde{C}_n^\pm$ are $\varepsilon_{n-1}$-close to $h(\R/\Z)$. Note that $\widetilde{C}_n^+$ and $\widetilde{C}_n^-$ are contained in $h(\R/\Z)^\Ext$ and $h(\R/\Z)^\Int$ respectively. Thus $h(\R/\Z)$ is accumulated by repelling cycles from both sides. This implies that $h(\R/\Z)$ is contained in the Julia set of $Q_{a',u'}$ and is not the boundary of any rotation domain of $Q_{a',u'}$.

Note that the conformal angle $\theta$ of two critical points with respect to the Herman ring of $Q_\alpha$ is chosen in $(0,2\pi)$. It follows that $g(\R/\Z)$ is not a circle by Schwarz reflection.
So neither is $h(\R/\Z)$ (if necessary by taking $\varepsilon$ smaller). Indeed, since $\gamma:=g(\R/\Z)$ is not a circle, we have
\begin{equation}
\delta(\gamma):=\inf_{a\in\gamma^\Int}\Big(\sup_{z\in\gamma}|a-z|-\inf_{z\in\gamma}|a-z|\Big)>0.
\end{equation}
Taking a small $0<\varepsilon<\min\{\delta(\gamma),1\}/10$, we have $\delta(\gamma')>0$ since $\gamma':=h(\R/\Z)$ is $\frac{\varepsilon}{1-\varepsilon}$-close to $\gamma$ in the Hausdorff metric. This implies that $\gamma'$ is not a circle.
Hence $h(\R/\Z)$ is a smooth degenerate Herman ring of $Q_{a',u'}$.
This completes the proof of Theorem \ref{thm:A-detail} and hence Theorem \ref{thm:main-1}.
\end{proof}

\begin{rmk}
(1) By \eqref{equ:alpha-a-u-est} and the definition of $\varepsilon_{n+1}$, we have $r_{\alpha'}<r_{\alpha(n)}+\varepsilon_{n}$. Since $r<r_{\alpha(n)}<r_{n}+\varepsilon_{n}<r+2\varepsilon_{n}$ for $n\geqslant 1$ and $\varepsilon_{n} \to 0$, we have $r_{\alpha(n)} \to r$ and hence $r_{\alpha'} \leqslant r$.
On the other hand, by the upper semicontinuity (see \eqref{equ:conf-radius}), we have $r_{\alpha'} \geqslant \lim _{n \to \infty} r_{\alpha(n)}$. Hence $r_{\alpha'}=r>0$ and $\alpha'$ is a Brjuno number.

\medskip
(2) Note that the smooth degenerate Herman ring $h(\R/\Z)$ does not contain any critical points of $Q_{a',u'}$ since an invariant Jordan curve cannot be smooth at both the critical point and the critical value. This can be also obtained from another observation: if $\varepsilon$ is small enough, the set of critical points of $Q_{a',u'}$ is sufficiently close to that of $Q_{a,u}$ while the latter has a definite positive distance to the invariant Jordan curve $g(\R/\Z)$.
\end{rmk}

\subsection{Proof of Theorem \ref{thm:main-2}}

For a Brjuno number $\alpha$, recall that $\Delta_\alpha$ is the Siegel disk of the quadratic polynomial $P_\alpha(z)=e^{2\pi\ii\alpha}z+z^2$.
In this subsection we prove a stronger result than Theorem \ref{thm:main-2}:

\begin{thm}\label{thm:Julia-positive}
For any $\varepsilon>0$ and bounded type $\alpha\in\HT_N$, there exist a Brjuno number $\alpha'\in\HT_N$ with $|\alpha-\alpha'|<\varepsilon$ and a cubic rational map $Q_{a',u'}$ with the following properties:
\begin{enumerate}
\item $Q_{a',u'}$ has a smooth degenerate Herman ring $\gamma'$ with rotation number $\alpha'$;
\item $\partial\Delta_{\alpha'}$ is smooth and $J(P_{\alpha'})$ has positive area;
\item There exists a quasiconformal mapping defined on $\EC$ which conjugates $Q_{a',u'}:\overline{(\gamma')^{\Ext}}\to\EC$ (resp. $Q_{a',u'}:\overline{(\gamma')^{\Int}}\to\EC$) to $P_{\alpha'}:\EC\setminus\Delta_{\alpha'}\to\EC$ (resp. $P_{-\alpha'}:\EC\setminus\Delta_{-\alpha'}\to\EC$); and
\item $J(Q_{a',u'})$ has positive area and $Q_{a',u'}$ has no irrationally indifferent periodic points, no Herman rings, and is not renormalizable.
\end{enumerate}
\end{thm}

\begin{proof}
(a) For any given $0<r<r_\alpha$ and $\theta\in(0,2\pi)$, let $(\alpha(n))_{n\geqslant 0}$ be the inductively defined sequence of bounded type numbers in the proof of Theorem \ref{thm:A-detail}.  Denote $\alpha':=\lim_{n\to\infty}\alpha(n)$. We still use the same notations as in the proof of Theorem \ref{thm:A-detail}. Then the cubic rational map $Q_{a',u'}$ has a smooth degenerate Herman ring $\gamma'=h(\R/\Z)$ with rotation number $\alpha'$.

\medskip
(b) We choose a sequence of numbers $(\varepsilon_{n}')_{n\geqslant 1}$ in $(0,1)$ such that $\prod_{n\geqslant 1}\left(1-\varepsilon_{n}'\right)>0$.
For $(\alpha(n))_{n\geqslant 0}$, by Lemmas \ref{lem:ABC04} and \ref{lem:ABC12}, one can require further that this sequence satisfies Part (a) and for all $n\geqslant 1$,
\begin{itemize}
\item $\area\left(L_{\alpha(n)}(r)\right) \geqslant\left(1-\varepsilon_{n}'\right) \area\left(L_{\alpha(n-1)}(r)\right)$; and
\item The real-analytic functions $t\mapsto\phi_{\alpha(n)}(r e^{2\pi\ii t})$ and $t\mapsto\phi_{\alpha(n-1)}(r e^{2\pi\ii t})$ are $\varepsilon_n'$-close in the Fr\'{e}chet space $C^\infty(\R/\Z,\C)$, where $\phi_{\alpha(n)}$ is the normalized linearizable map satisfying \eqref{equ:linearization}.
\end{itemize}

By \cite[p.\,5]{ABC04}, $\partial\Delta_{\alpha'}$ is smooth.
By Remark (1) after Theorem \ref{thm:A-detail}, we have $r_{\alpha'}=r$, $\Delta_{\alpha'}(r)=\Delta_{\alpha'}$ and $L_{\alpha'}(r)=J(P_{\alpha'})$.
According to \cite[p.\,735]{BC12}, we have $\limsup\limits_{n\to\infty} L_{\alpha(n)}(r)\subset L_{\alpha'}(r)=J(P_{\alpha'})$ and hence
\begin{equation}
\area\big(J(P_{\alpha'})\big) \geqslant \area\Big(\limsup_{n\to\infty} L_{\alpha(n)}(r)\Big) \geqslant
\area\big(L_{\alpha}(r)\big) \cdot \prod_{n\geqslant 1}(1-\varepsilon_{n}')>0.
\end{equation}

Then there exists an real-analytic map $\xi_\alpha:\T\to\gamma$ which conjugates the irrational rotation $\zeta\mapsto e^{2\pi\ii\alpha} \zeta$ to $Q_\alpha: \gamma\to \gamma$.

(c) We only consider $(\gamma')^\Ext$ since the proof for $(\gamma')^\Int$ is analogous. By \eqref{equ:smooth-conj}, the smooth map $\chi_{\alpha'}:\T\to\gamma'$ conjugates the irrational rotation $\zeta\mapsto e^{2\pi\ii\alpha'} \zeta$ to $Q_{a',u'}: \gamma'\to \gamma'$.
Hence $\chi_{\alpha'}^{-1}:\gamma'\to\T$ can be extended continuously to $\Theta:\overline{(\gamma')^{\Int}}\to\overline{\D}$ such that $\Theta:(\gamma')^{\Int}\to\D$ is quasiconformal (see \cite[Proposition 2.30(a)]{BF14a}).
Define
\begin{equation}\label{equ:G-quasiregular}
G(z):=
\left\{
\begin{array}{ll}
Q_{a',u'}(z)  &~~~~~~~\text{if}~z\in \overline{(\gamma')^{\Ext}}, \\
\Theta^{-1}\big(e^{2\pi\ii\alpha'}\Theta(z)\big) &~~~~~~\text{if}~z\in(\gamma')^{\Int}.
\end{array}
\right.
\end{equation}
Then $G:\EC\to\EC$ is a quasi-regular map.
By a similar proof to Corollary \ref{cor:zero-area}, there exists a quasiconformal mapping $\Upsilon:\EC\to\EC$ such that $\Upsilon\circ G\circ\Upsilon^{-1}=P_{\alpha'}$. In particular, $\Upsilon$ conjugates $Q_{a',u'}:\overline{(\gamma')^{\Ext}}\to\EC$ to $P_{\alpha'}:\EC\setminus\Delta_{\alpha'}\to\EC$.

\medskip
(d)
By Parts (b) and (c), we conclude that $J(Q_{a',u'})$ has positive area since quasiconformal mappings are absolutely continuous with respect to $2$-dimensional Lebesgue measure.
We prove the rest statement by contradiction. 

Suppose $Q_{a',u'}$ has a periodic rotation domain $W$. Without loss of generality, we assume that $W\subset (\gamma')^{\Ext}$. Since $\gamma'$ is a smooth degenerate Herman ring, by Part (c),
\begin{equation}\label{equ:MV}
\MV:=(\gamma')^{\Ext}\cap\Big(\bigcup_{k\geqslant 1} G^{-k}\big((\gamma')^{\Int}\big)\Big)
\end{equation}
is disjoint from the post-critical set $\MP(Q_{a',u'})$ of $Q_{a',u'}$ and $\partial{W}\subset \MV$.
However, by \cite[Theorem 11.17 and Lemma 15.7]{Mil06}, $\partial{W}$ is contained in $\MP(Q_{a',u'})$, which is a contradiction.

Suppose $Q_{a',u'}$ has a Cremer cycle $\MO$. Obviously $\MO\cap\gamma'=\emptyset$. Without loss of generality, we assume that $\MO\cap(\gamma')^{\Ext}\neq\emptyset$. The argument is divided into two cases. If $\MO\cap (\gamma')^\Int\neq\emptyset$, then $\MO\cap\MV\neq\emptyset$. This is impossible since each point of $\MO$ is contained in the post-critical set $\MP(Q_{a',u'})$ \cite[Theorem 11.17]{Mil06}.
If $\MO\cap (\gamma')^\Int=\emptyset$, then by Part (c), $P_{\alpha'}$ has two disjoint indifferent periodic cycles, which contradicts the Fatou-Shishikura inequality \cite[Corollary 1]{Shi87}.

By definition, see \cite[\S 7]{McM94b} or \cite{DH85b}, $Q_{a',u'}$ is renormalizable if there exist an integer $p\geqslant 1$ and two Jordan domains $U$, $V$ such that
\begin{itemize}
\item $U$ is compactly contained in $V$ and $Q_{a',u'}^{\circ p}:U\to V$ is a proper holomorphic map of degree at least two; and
\item The small filled-in Julia set $K_0:=\bigcap_{k\geqslant 0}Q_{a',u'}^{-kp}(U)$ is connected and contains a critical point of $Q_{a',u'}$.
\end{itemize}
Suppose $Q_{a',u'}$ is renormalizable as above. Since $Q_{a',u'}^{\circ p}:\gamma'\to\gamma'$ is conjugate to an irrational rotation, we have $\gamma'\cap K_0=\emptyset$ or $\gamma'\subset K_0$.
If $\gamma'\subset K_0$, then $U$ must contain $(\gamma')^{\Int}$ or $(\gamma')^{\Ext}$. In either case we have $V=\EC$, which is a contradiction. Hence we have $\gamma'\cap K_0=\emptyset$.
Without loss of generality, we assume that 
\begin{equation}
\MK\cap(\gamma')^{\Ext}\neq\emptyset, \text{\quad where }\MK=\bigcup_{i=0}^{p-1}Q_{a',u'}^{\circ i}(K_0).
\end{equation}
The argument is similar to the Cremer case. If $\MK\cap (\gamma')^\Int\neq\emptyset$, then $\MK\cap\MV\neq\emptyset$, where $\MV$ is defined in \eqref{equ:MV}. This is impossible since $\MK\cap \MP(Q_{a',u'})\neq\emptyset$.
If $\MK\cap (\gamma')^\Int=\emptyset$, then by Part (c), $P_{\alpha'}$ is renormalizable, which is a contradiction.
The proof is complete and hence Theorem \ref{thm:main-2} holds.
\end{proof}

\begin{rmk}
(1) As an immediate corollary of Theorem \ref{thm:Julia-positive}(c), the post-critical set $\PC(Q_{a',u'})$ of $Q_{a',u'}$ in $\C\setminus\{0\}$ is the disjoint union of the smooth degenerate Herman ring $\gamma'$ and the quasiconformally homeomorphic images of $\PC(P_{\alpha'})\setminus\partial\Delta_{\alpha'}$ and $\PC(P_{-\alpha'})\setminus\partial\Delta_{-\alpha'}$.

\medskip
(2) The topology of $\PC(P_\alpha)$ is characterized in \cite{Che22b} for any high type number $\alpha$. In particular, in the case of Theorem \ref{thm:Julia-positive}, $\PC(P_{\alpha'})$ is a \textit{one-sided hairy Jordan curve}. Note that $Q_{a',u'}$ is the limit of $Q_{\alpha(n),r,\theta}$ as $n\to\infty$. The parameter $\theta\in(0,2\pi)$ measures how the two one-sided hairy Jordan curves are pasted together.
\end{rmk}

\section{Conclusion}

In this section we supplement some results and list some problems about (degenerate) Herman rings.

\subsection{Smooth Herman rings vs positive area}

Note that Theorem \ref{thm:Julia-positive}(b) indicates the existence of quadratic Siegel polynomials having a Julia set with positive area and a Siegel disk with smooth boundary.
A fixed Herman ring of a rational map is called \textit{non-symmetric} if it does not contain any invariant circle.
By a slight modification of the proofs of Theorems \ref{thm:A-detail} and \ref{thm:Julia-positive}, we have the following result.

\begin{cor}\label{cor:HR-smooth}
There exist cubic rational maps having a non-symmetric smooth Herman ring and a Julia set with positive area.
\end{cor}

Indeed, the only difference from Theorems \ref{thm:A-detail} and \ref{thm:Julia-positive} is that one should choose a decreasing sequence $(r_{n})_{n\geqslant 0}$ satisfying $r_n\to \widetilde{r}\in (r, r_\alpha)$ (note that $r_{n}\to r$ in Theorem \ref{thm:A-detail}). Since the rest place is completely the same, we omit the details.

\subsection{Cubic Blaschke products}\label{subsec:Blaschke}

Let $a>3$ be given. In \cite{Avi03}, Avila proved the existence of (symmetric) smooth Herman rings in the family of cubic Blaschke products:
\begin{equation}\label{equ:Blaschke}
f_t(z):=e^{2\pi\ii t}z^2\frac{z-a}{1-az}, \text{\quad where } t\in\R.
\end{equation}
In his proof the quasiconformal surgery is not needed since  $f_t$ serves the family with symmetric Herman rings containing the invariant unit circle automatically.

Similar to Corollary \ref{cor:HR-smooth} (based on Theorems \ref{thm:A-detail} and \ref{thm:Julia-positive}), we have

\begin{cor}\label{cor:HR-smooth-B}
There exists $t$ such that $f_t$ has a smooth Herman ring and a Julia set with positive area.
\end{cor}

Similar to the proof of Theorem \ref{thm:Julia-positive}, we have

\begin{cor}\label{cor:HR-area}
There exists $t$ such that $f_t$ has a nowhere dense Julia set of positive area, and that $f_t$ has no irrationally indifferent periodic points, no Herman rings, and is not renormalizable.
\end{cor}

To find explicit formulas of rational maps having degenerate Herman rings, in this paper we use the cubic families $Q_{\alpha,r,\theta}$ and $Q_{a,u}$ (see Lemma \ref{lem:rigidity}).
If we fix the conformal angle $\theta:=\pi$ between two critical points (see \S\ref{subsec:rigidity}), then one can use \textit{antipode preserving} cubic maps to replace $Q_{a,u}$ with (see \cite{BBM18}):
\begin{equation}
h_q(z)=z^2\frac{q-z}{1+\overline{q} z}, \text{\quad where }q\in\C.
\end{equation}

\subsection{Some problems}

From Theorem \ref{thm:A-detail}, we know that the rotation numbers of smooth degenerate Herman rings constructed in this paper are all Brjuno type. Hence a natural problem is:

\begin{prob}
Is there a smooth degenerate Herman ring whose rotation number is not of Brjuno type?
\end{prob}

Besides smooth Siegel disks, Buff and Ch\'{e}ritat also constructed some quadratic Siegel disks whose boundaries have various degrees of regularity in \cite{BC07}. In particular, they proved the existence of quadratic Siegel disk boundaries which are Jordan curves containing no critical point but not quasi-circles. Then we have:

\begin{prob}
Is there a degenerate Herman ring which is not a quasi-circle?
\end{prob}

The question of Eremenko in \cite[Question 1.7, p.\,2]{Ere20} was asked not only for rational maps, but also for transcendental meromorphic functions:

\begin{prob}\label{prob:trans}
Are there smooth degenerate Herman rings for transcendental meromorphic functions?
\end{prob}

From the results mentioned in this paper, it is hard to obtain an answer to Problem \ref{prob:trans} directly.
Indeed, for this problem, one obstacle is to control the conformal radii of Siegel disks of two different maps at the same time (one of them is transcendental), and the second is to establish a result about the control of the loss of area of the sets which are similar to filled-in Julia sets (see Lemmas \ref{lem:ABC12} and \ref{lem:control-loss}).

\bibliographystyle{amsalpha}
\bibliography{E:/Latex-model/Ref1}

\end{document}